\documentclass[12pt]{amsart}
\usepackage{a4wide,enumerate,xcolor}
\usepackage{amsmath,graphicx,comment}
\allowdisplaybreaks

\let\pa\partial
\let\na\nabla
\let\eps\varepsilon
\newcommand{\N}{{\mathbb N}}
\newcommand{\R}{{\mathbb R}}
\newcommand{\diver}{\operatorname{div}}

\newtheorem{theorem}{Theorem}
\newtheorem{lemma}[theorem]{Lemma}

\newtheorem{remark}[theorem]{Remark}

\newtheorem{definition}{Definition}

\begin{document}

\title[A nonlocal Busenberg--Travis system]{A nonlocal Busenberg--Travis cross-diffusion system \\ with nonlinear Brinkman law}

\author[P. Hirvonen]{Peter Hirvonen}
\address{Institute of Analysis and Scientific Computing, TU Wien, Wiedner Hauptstra\ss e 8--10, 1040 Wien, Austria}
\email{peter.hirvonen@tuwien.ac.at} 

\author[A. J\"ungel]{Ansgar J\"ungel}
\address{Institute of Analysis and Scientific Computing, TU Wien, Wiedner Hauptstra\ss e 8--10, 1040 Wien, Austria}
\email{juengel@tuwien.ac.at} 

\date{\today}

\thanks{The authors acknowledge partial support from the Austrian Science Fund (FWF), grant 10.55776/PAT2687825, and from the Austrian Federal Ministry for Women, Science and Research and implemented by \"OAD, project MultHeFlo. This work has received funding from the European Research Council (ERC) under the European Union's Horizon 2020 research and innovation programme, ERC Advanced Grant NEUROMORPH, no.~101018153. For open-access purposes, the authors have applied a CC BY public copyright license to any author-accepted manuscript version arising from this submission.} 

\begin{abstract}
A nonlocal Busenberg--Travis cross-diffusion system for segregating populations is analyzed in a bounded domain with no-flux boundary conditions. The velocities of the species solve a regularized Darcy law, which can be interpreted as a Brinkman equation. Compared to results in the literature, the density--pressure relation is assumed to be nonlinear. The global existence of weak solutions to this system is shown for a broad range of the exponents of the power-law nonlinearity, and the localization limit is proved. The proofs are based on uniform estimates coming from the Tsallis entropy inequality. Due to regularity issues, the original problem is approximated by various schemes, and the de-regularization limits are obtained through compactness arguments. 
\end{abstract}

 

\maketitle


\section{Introduction}

The Busenberg--Travis model has emerged as a fundamental framework for capturing the spatial-temporal evolution of segregating populations with nonlinear interaction mechanisms \cite{BuTr83}. According to Darcy's law, the partial velocity is given by the gradient of the pressure. Grindrod \cite{Gri88} replaced Darcy's law by Brinkman's law, which augments Darcy's description by a diffusion term accounting for viscous stresses \cite{Bri49}. This modification captures intermediate regimes between purely Darcy and Stokes flows. The resulting system consists of nonlocal cross-diffusion equations. While a linear dependence of the pressure functions on the partial densities was suggested and analyzed in \cite{JVZ24}, we consider here a nonlinear power-law dependence. For this model, we establish for the first time the global existence of weak solutions and rigorously justify the localization limit.

\subsection{Model equations}

The Busenberg--Travis equations with Brinkman-type law for the population densities $u_i(x,t)\ge 0$ and the velocities $v_i(x,t)\in\R^d$ read as
\begin{align}
  \pa_t u_i - \sigma_i\Delta u_i + \diver(u_iv_i) &= 0 \quad
  \mbox{in }\Omega,\ t>0, \label{1.u} \\
  -\eps\Delta v_i + v_i &= -\na p_i(u), \quad  i=1,\ldots,n,
  \label{1.v}
\end{align}
where $\Omega\subset\R^d$ ($d\ge 1$) is a bounded domain, $u=(u_1,\ldots,u_n)$ is the solution vector, $\sigma_i>0$ are the diffusivities, and the parameter $\eps$ distinguishes between the Darcy law ($\eps=0$) and the Brinkman law ($\eps>0$). The pressure function $p_i(u)$ is given by
\begin{align}
  p_i(u) = \sum_{j=1}^n a_{ij}u_j^\beta, \label{1.p}
\end{align}
where $\beta>0$ and $a_{ij}\in\R$, and we impose initial and no-flux boundary conditions
\begin{align}\label{1.bic}
  u_i(0) = u_i^0 \ \mbox{in }\Omega, \quad
  (\sigma_i\na u_i + u_iv_i)\cdot\nu = 0, \ 
  v_i = 0\ \mbox{on }\pa\Omega,\ i=1,\ldots,n.
\end{align}
The model consists of the mass conservation equations \eqref{1.u}, the Brinkman equations \eqref{1.v}, and the nonlinear density--pressure relations \eqref{1.p}. While a linear dependence between the pressures and densities ($\beta=1$) corresponds to ideal fluid mixtures, a nonlinear pressure dependence accounts for non-ideal phenomena. The existence of global {\em bounded} weak solutions was proved in \cite{JVZ24} in one space dimension only. A mathematical motivation for studying system \eqref{1.u}--\eqref{1.bic} is to investigate whether bounded solutions exist in higher space dimensions for $\beta\neq 1$. 

\subsection{State of the art}

The original two-species Busenberg--Travis system with $a_{ij}=1$ and $\eps=0$ was suggested in \cite{BuTr83} to describe the segregation of population species. It was analyzed in, e.g., \cite{BGHP85,BHIM12}. The analysis was extended to an arbitrary number of species and $a_{ij}=a_j$ in \cite{DHJ23}. The model was rigorously derived from a moderately interacting particle system using a mean-field approach \cite{CDJ19}. The density-pressure relation $p_i(u)=(\sum_{j=1}^n a_ju_j)^m$ with $m\ge 1$ was investigated in \cite{DrJu20,LLP17}. In the works \cite{CSS26,GHLP22,JPZ24}, nonlocal relations of the type $p_i(u)=\sum_{j=1}^n K_{ij}*u_j$ have been analyzed. 

Grindrod added a diffusion term with parameter $\eps>0$ to smooth sharp spatial variations in $\na p_i(u)$; see equation \eqref{1.v}. The system with linear density--pressure relations was formally derived from a kinetic model in \cite{JPT25}. We present the derivation from a kinetic system leading to the nonlinear Brinkman equations \eqref{1.v} in Appendix \ref{sec.model}. The existence and stability of
uniform steady states was shown in \cite{KKMGV17}, and in \cite{JVZ24}, the global existence of weak solutions to the model with linear density--pressure relations was proved, assuming that the matrix $(a_{ij})$ is positive definite.

The localization limit $\eps\to 0$ in system \eqref{1.u}--\eqref{1.p} with linear density--pressure relation was proved in \cite{JPZ24}. For related localization limits in the kernel function associated to nonlocal Busenberg--Travis systems, we refer to, e.g., \cite{BuEs23,CEW24,DHPP24}. To the best of our knowledge, the existence analysis and the nonlocal-to-local limit for system \eqref{1.u}--\eqref{1.bic} are new.

\subsection{Mathematical tools and strategy}

The analysis of the model with linear density--pressure dependence was performed by using the entropy method, with the Boltzmann entropy density $h_1(u)=\sum_{i=1}^n u_i(\log u_i-1)$. Due to the power-law nonlinearity in \eqref{1.v}, this entropy functional cannot be used here. Instead we consider the Tsallis entropy
\begin{align*}
  H(u) = \int_\Omega h(u)dx 
  = \sum_{i=1}^n\int_\Omega\frac{u_i^\beta-u_i}{\beta-1}dx,
\end{align*}
which is a convex functional for any $\beta>0$, $\beta\neq 1$. In the limit $\beta\to 1$, we find that $(u_i^\beta-u_i)/(\beta-1)\to u_i\log u_i$. To derive the entropy inequality, which provides a priori estimates, we need some notation. Let $v_i=L_\eps(-\na p_i(u))$ be the solution to \eqref{1.v} with homogeneous Dirichlet boundary conditions. The operator $L_\eps$ on $H^{1}(\Omega;\R^d)'$ admits the square root operator $K_\eps$ satisfying $K_\eps\circ K_\eps = L_\eps$. Assuming that $(a_{ij})$ is positive definite with smallest eigenvalue $\alpha>0$, a formal computation shows that, along solutions to \eqref{1.u}--\eqref{1.bic}, 
\begin{align}\label{1.ei}
  \frac{d}{dt}H(u) + \frac{4}{\beta}\sum_{i=1}^n\sigma_i\int_\Omega
  |\na u_i^{\beta/2}|^2 dx + \alpha\sum_{i=1}^n
  \int_\Omega|K_\eps(\na u_i^\beta)|^2 dx \le 0, \quad t\in(0,T).
\end{align}
This gives bounds for $u_i$ in $L^\infty(0,T;L^{\max\{1,\beta\}}(\Omega))$, $u_i^{\beta/2}$ in  $L^2(0,T;H^1(\Omega))$, and $v_i$ in $L^2(\Omega\times(0,T);\R^d)$ (the last bound follows from \eqref{2.Leps} below). The analysis of our model depends on $\beta$, and we distinguish four cases:
\begin{itemize}
\item $0<\beta<1/d$, $d\ge 1$: Mass conservation implies that $u_i^\beta$ is bounded in $L^\infty(0,T;L^{1/\beta}(\Omega))$, and we deduce from elliptic regularity \cite{Gro89} that $v_i$ is bounded in $L^\infty(0,T;W^{1,1/\beta}(\Omega;$ $\R^d))$, which embeds continuously into $L^\infty(\Omega_T;\R^d)$, where $\Omega_T=\Omega\times(0,T)$. This allows us to apply the Alikakos iteration technique as in \cite{JVZ24} to conclude that $u_i$ is bounded in $L^\infty(0,T;L^\infty(\Omega))$. In this situation, we obtain a {\em bounded} weak solution.
\item $\beta=1/2$, $d=2$: In this case, we can prove estimates only in non-reflexive spaces, and for this reason, the existence of just integrable solutions can be guaranteed; see Remark \ref{rem.12}.
\item $1/2<\beta<1$, $d\le 2$: Since we cannot obtain bounded velocities $v_i$, this case is more delicate than the case $\beta<1/d$. We only obtain bounds for $u_i$ in $L^{\beta+1}(\Omega_T)$ and $v_i$ in $L^r(\Omega_T;\R^d)$ for some $r>4$. For this case, an elliptic regularity result is assumed to be satisfied, see Assumption (A4) below.
\item $\max\{1,2d/(d+2)\}<\beta<2$, $d\ge 1$: The entropy inequality yields a bound for $u_i$ in $L^\infty(0,T;L^\beta(\Omega))$, which is better than pure mass conservation. Together with inequality \eqref{1.ei}, this yields bounds for $u_i$ in $L^{\beta(d+2)/d}(\Omega_T)$ and for $v_i$ in $L^2(\Omega_T;\R^d)$. 
\item $\beta\ge 2$, $d\ge 1$: We obtain the same regularity for $u_i$ and $v_i$ as in the previous case, but $\na u_i$ is no longer bounded in any Lebesgue space because of the degeneracy. Instead, by the entropy inequality, the expression $\na u_i^{\beta/2}$ is bounded in $L^2(\Omega_T)$. 
\end{itemize}

Summarizing, we prove a global existence result for all $0<\beta<\infty$ if $d\le 2$ and for all $0<\beta<1/d$, $2d/(d+2)<\beta<\infty$ if $d\ge 3$. The case $\beta=1$ is covered in \cite{JVZ24}. The case $1/d\le\beta\le 2d/(d+2)$ for $d\ge 3$ is open, since the elliptic regularization \eqref{1.v} is not strong enough to provide suitable estimates.

To make our results rigorous, we need to regularize equations \eqref{1.u} and \eqref{1.v} in a suitable way. The regularization depends on the value of $\beta$. As the parabolic problem may lead to issues with respect to the time regularity, we approximate \eqref{1.u} by an implicit Euler scheme with time step size $\tau>0$. When $\beta>1/2$, we approximate equation \eqref{1.v} with a regularized solution operator $L_\eps^\eta$ on $H^m(\Omega;\R^d)'$ with parameter $\eta>0$ such that $L_\eps^\eta$ maps into $H^m(\Omega;\R^d)\hookrightarrow L^\infty(\Omega;\R^d)$ if $m>d/2$. Then the solution $v_i=L_\eps^\eta(-\na p_i(u))$ is bounded in $L^\infty(\Omega;\R^d)$. This bound may blow up in the limit $\eta\to 0$. Because of regularity issues, the limit $(\eta,\tau)\to 0$ can only be performed in two space dimensions. The difficulty is to show that $u_iv_i$ is integrable. A further difficulty arises in the case $\beta\ge 2$, since $v_i=L_\eps^\eta(-\na p_i(u))$ may not exist if $p_i(u)\not\in L^1(\Omega)$. Therefore, we approximate $u_i^\beta$ by introducing a truncation operator $S_N^\beta(u_i)$ with parameter $N>0$. Then $S_N^\beta(u_i)\in L^\infty(\Omega)$ and $v_i\in L^\infty(\Omega;\R^d)$. Hence, we need to pass to the limit in the three parameters $N\to\infty$ and $(\eta,\tau)\to 0$. This gives the global existence of weak solutions.

A second result concerns the localization limit $\eps\to 0$. As this limit is singular, we expect to lose some a priori bounds, in particular those bounds for $v_i$ that are based on elliptic regularity. This issue is overcome by observing that the estimates from the entropy inequality \eqref{1.ei} are independent of $\eps$. These estimates are not sufficient to cover all ranges of $\beta$, since we can at best expect a bound for $v_i$ in $L^2(\Omega_T)$. Thus, to define the product $u_iv_i$, we need a bound for $u_i$ in some space $L^\rho(\Omega_T)$ with $\rho>2$. For this reason, we need small space dimensions or large values of $\beta$ to perform this limit rigorously.

\subsection{Main results}

In the following, we present the precise main results. First, we summarize our assumptions:

\begin{itemize}
\item[(A1)] Domain: $\Omega\subset\R^d, d\ge 1$, is a bounded domain with smooth boundary. 
\item[(A2)] Parameters: $T>0$, $\sigma_1,\ldots,\sigma_n>0$, $\eps>0$, $(a_{ij})_{i,j=1}^n \subset \R^{n\times n}$ is a positive definite matrix with smallest eigenvalue $\alpha>0$. 
\item[(A3)] Initial Data: $u_i^0 \in L^\infty(\Omega)$, $u_i^0(x)>0$ for a.e.\ $x\in \Omega$.
\item[(A4)] Elliptic Regularity: For $f\in L^p(\Omega;\R^d)$, $p\in(1,\infty)$, and $m>1$, the higher-order elliptic problem
\begin{equation}\label{1.A4}
  \int_\Omega\bigg(\eta\sum_{|\alpha|=m}D^\alpha v D^\alpha\phi
  + \na v\cdot\na\phi + v\phi\bigg)dx = -\int_\Omega f\diver\phi dx
\end{equation}
for $\phi\in H^m(\Omega)\cap H_0^1(\Omega)$ has a solution $v\in H^m(\Omega)\cap H_0^1(\Omega)$ that satisfies $\|v\|_{W^{1,p}(\Omega)}\leq C$, where $C>0$ does not depend on $\eta>0$.
\end{itemize}

The symbol $D^\alpha$ in Assumption (A4) stands for the partial derivative of order $|\alpha|$, where $\alpha\in\N^d$ is a multiindex. This condition is only needed for the case $\beta\in (1/2,1)$ and $d=2$; see Theorem \ref{thm.ex2}. We briefly discuss the validity of this assumption in Appendix \ref{sec:appEllReg}.

\begin{definition}[Weak solution]
We say that $u$ is a (nonnegative) {\em weak solution} to \eqref{1.u}--\eqref{1.bic} on $(0,T)$ if $u_i\ge 0$ a.e.\ in $\Omega_T=\Omega\times(0,T)$ and
\begin{align*}
  & u_i\in L^1(0,T;W^{1,1}(\Omega)), \quad 
  \pa_t u_i\in L^1(0,T;W^{1,\infty}(\Omega)'), \\
  & v_i=L_\eps(-\na p_i(u))\in L^2(\Omega_T;\R^d), 
  \quad u_iv_i\in L^1(\Omega_T;\R^d),
\end{align*}
it holds for all $\phi_i\in L^\infty(0,T;W^{1,\infty}(\Omega))$,
\begin{align*}
  \int_0^T\langle\pa_t u_i,\phi_i\rangle dt
  + \sigma_i\int_0^T\int_\Omega\na u_i\cdot\na\phi_i dxdt
  = \int_0^T\int_\Omega u_iv_i\cdot\na\phi_i dxdt,
\end{align*}
and the initial condition in \eqref{1.bic} is satisfied in the sense of $W^{1,\infty}(\Omega)'$. 
\end{definition}

\begin{definition}[Very weak solution]
We say that $u$ is a (nonnegative) {\em very weak solution} to \eqref{1.u}--\eqref{1.bic} on $(0,T)$ if $u_i\ge 0$ a.e.\ in $\Omega_T$ and
\begin{align*}
  & u_i\in L^{\beta}(\Omega_T), \quad 
  \pa_t u_i\in L^1(0,T;W^{2,\infty}(\Omega)'), \\
  & v_i=L_\eps(-\na p_i(u))\in L^2(\Omega_T;\R^d), 
  \quad u_iv_i\in L^1(\Omega_T;\R^d),
\end{align*}
it holds for all $\phi_i\in L^\infty(0,T;W^{2,\infty}(\Omega))$ satisfying $\na\phi_i\cdot\nu=0$ on $\pa\Omega$,
\begin{align*}
  \int_0^T\langle\pa_t u_i,\phi_i\rangle dt
  - \sigma_i\int_0^T\int_\Omega u_i\Delta\phi_i dxdt
  = \int_0^T\int_\Omega u_iv_i\cdot\na\phi_i dxdt,
\end{align*}
and the initial condition in \eqref{1.bic} is satisfied in the sense of $W^{2,\infty}(\Omega)'$. 
\end{definition}

We state now the existence results for the various ranges of $\beta$. We always assume that Assumptions (A1)--(A3) hold.

\begin{theorem}[Global existence for $\beta<1/d$]\label{thm.ex1}
Let $0<\beta<1/d$, $d\ge 1$. Then there exists a nonnegative weak solution $u=(u_1,\ldots,u_n)$ to \eqref{1.u}--\eqref{1.bic} satisfying 
\begin{equation}\label{1.reg1}
\begin{aligned}
  & u_i\in L^\infty(0,T;L^\infty(\Omega)), \quad 
  \pa_t u_i\in L^2(0,T;H^{1}(\Omega)'), \\
  & u_i^{\beta/2},\ u_i\in L^2(0,T;H^1(\Omega)), \quad
  v_i\in L^\infty(\Omega_T;\R^d),
\end{aligned}
\end{equation}
and the initial data are satisfied a.e.\ in $\Omega$. The solution $(u_i,v_i)$ is unique in the class of solutions satisfying \eqref{1.reg1}. 
\end{theorem}

\begin{theorem}[Global existence for $1/2<\beta<1$]
\label{thm.ex2}
Let $1/2<\beta<1$, $d\leq2$, and let additionally Assumption (A4) be satisfied. Then there exists a nonnegative weak solution $u=(u_1,\ldots,u_n)$ to \eqref{1.u}--\eqref{1.bic} satisfying, with $q=\frac23(\beta+1)\ge 1$ and some $r>4$,
\begin{align*}
  & u_i\in L^\infty(0,T;L^1(\Omega))\cap 
  L^{\beta+1}(\Omega_T)\cap L^q(0,T;W^{1,q}(\Omega)), \quad 
  \pa_t u_i\in L^1(0,T;W^{1,\infty}(\Omega)'), \\
  & u_i^{\beta/2}\in L^2(0,T;H^1(\Omega)), \quad
  v_i\in L^{(\beta+1)/\beta}(0,T;L^\infty(\Omega;\R^d))
  \cap L^r(\Omega_T;\R^d). 
\end{align*}
\end{theorem}

The case $\beta=1/2$, $d=2$ is discussed in Remark \ref{rem.12} below.

\begin{theorem}[Global existence for $2d/(d+2)<\beta<2$]\label{thm.ex3}
Let $\max\{1,2d/(d+2)\}<\beta<2$, $d\ge 1$. Then there exists a nonnegative weak solution $u=(u_1,\ldots,u_n)$ to \eqref{1.u}--\eqref{1.bic} satisfying, with $q=\beta(d+2)/(\beta+d)\ge 1+d/(d+4)$,
\begin{align*}
  & u_i\in L^\infty(0,T;L^\beta(\Omega))\cap
  L^{\beta(d+2)/d}(\Omega_T)\cap L^q(0,T;W^{1,q}(\Omega)), \quad 
  \pa_t u_i\in L^1(0,T;W^{1,\infty}(\Omega)'), \\
  & u_i^{\beta/2}\in L^2(0,T;H^1(\Omega)), \quad
  v_i\in L^2(\Omega_T;\R^d).
\end{align*}
\end{theorem}

\begin{theorem}[Global existence for $\beta\ge 2$]\label{thm.ex4}
Let $\beta\ge 2$, $d\ge 1$. Then there exists a nonnegative very weak solution $u=(u_1,\ldots,u_n)$ to \eqref{1.u}--\eqref{1.bic} satisfying 
\begin{align*}
  & u_i\in L^\infty(0,T;L^\beta(\Omega))\cap
  L^{\beta(d+2)/d}(\Omega_T), \quad 
  \pa_t u_i\in L^1(0,T;W^{2,\infty}(\Omega)'), \\
  & u_i^{\beta/2}\in L^2(0,T;H^1(\Omega)), \quad
  v_i\in L^2(\Omega_T;\R^d).
\end{align*}
\end{theorem}

Finally, we state the result for the singular limit $\eps\to 0$.

\begin{theorem}[Localization limit]\label{thm.loc}
Let $(u^{(\eps)},v^{(\eps)})$ be a nonnegative weak or very weak solution to \eqref{1.u}--\eqref{1.bic} satisfying the entropy inequality \eqref{1.ei}. Let
\begin{align*}
  0<\beta<1 \mbox{ if }d=1, \quad 1<\beta<2 \mbox{ if }d\leq 2, \quad
  \mbox{or}\quad\beta\ge 2\mbox{ if }d\ge 1.
\end{align*} Then there exists a subsequence (not relabeled) such that, for some $\rho>2$,
\begin{align*}
  u_i^{(\eps)}\to u_i\quad\mbox{strongly in }L^\rho(\Omega_T), \quad
  v_i^{(\eps)}\rightharpoonup v_i \quad\mbox{weakly in }L^2(\Omega_T),
\end{align*}
and $(u,v)$ solves
\begin{align*}
  \pa_t u_i - \diver(\sigma_i \na u_i + u_i\na p_i(u)) = 0
  \quad\mbox{in }\Omega,\ t>0,\ i=1,\ldots,n,
\end{align*}
and $u_i(0)=u_i^0$ in the sense of $W^{2,\infty}(\Omega)'$.
\end{theorem}

\begin{remark}[Case $\beta=1/2$]\label{rem.12}\rm
The case $\beta=1/2$, $d=2$ is surprisingly delicate, since we obtain uniform bounds for $|\na u_i|$ only in $L^1(\Omega_T)$ and for $\pa_t u_i$ only in $L^1(0,T;W^{1,\infty}(\Omega)')$, which are non-reflexive spaces. Still, we can apply the Aubin--Lions lemma to conclude the a.e.\ convergence of an approximate sequence of $u_i$, and it is possible to prove the existence of an integrable solution in the sense
\begin{align*}
  -\int_0^T\int_\Omega u_i\pa_t\phi_i dxdt 
  - \int_\Omega u_i^0\phi_i(0)dx
  - \sigma_i\int_0^T\int_\Omega u_i\Delta\phi_i dxdt
  = \int_0^T\int_\Omega u_i v_i\cdot\na\phi_i dxdt
\end{align*}
for suitable test functions $\phi_i$; see Section \ref{sec.12}. 
\end{remark}

\begin{remark}[Reaction rates]\rm
The original model of Grindrod \cite{Gri88} contains reaction terms. In the model of  \cite{JVZ24}, Lotka--Volterra terms of the form $u_if_i(u)$ have been included, yielding additional bounds. Indeed, let $f_i(u)=b_{i0}-\sum_{j=1}^n b_{ij}u_j$ with $b_{i0}\ge 0$, $b_{ij}\ge 0$ for $i\neq j$, and $b_{ii}>0$ for $i,j=1,\ldots,n$. Using the test function $1-\beta u_i^{\beta-1}$ (if $\beta<1$) or $\beta u_i^{\beta-1}-1$ (if $\beta>1$) in the weak formulation of \eqref{1.u}, we can estimate
\begin{align*}
  u_if_i(u) \le -C_0u_i^{\max\{1,\beta\}+1} + C_1,
\end{align*}
where $C_0$, $C_1>0$. This yields an improved bound for $u_i$ in $L^\infty(0,T;L^2(\Omega))$ instead of $L^\infty(0,T;L^1(\Omega))$ (if $\beta<1$) and in $L^\infty(0,T;L^{\beta+1}(\Omega))$ instead of $L^\infty(0,T;L^\beta(\Omega))$ (if $\beta>1$). In particular, we can improve the existence results allowing for $1<\beta<2$, since the reaction rates yield an estimate for $u_i$ in $L^r(\Omega_T)$ with $r=\beta+1>2$. Our results can be extended to such a situation, though deriving the approximate entropy inequality is more complex. We have chosen to omit the reaction rates in order to focus on the nonlocal and diffusive effects. 
\end{remark}

The paper is organized as follows. Some auxiliary results on the operators $L_\eps$ and $L_\eps^\eta$ are given in Section \ref{sec.prep}. Sections \ref{sec.beta1}--\ref{sec.beta4} are concerned with the proofs of Theorems \ref{thm.ex1}--\ref{thm.ex4}, respectively. The localization limit in Theorem \ref{thm.loc} is shown in Section \ref{sec.loc}. In Section \ref{sec.num}, we present some numerical simulations that explore how the solution behavior depends on the parameter $\beta$. We formally derive the system \eqref{1.u}--\eqref{1.p} from a kinetic model by using a Chapman--Enskog expansion in Appendix \ref{sec.model}. Finally, in Appendix \ref{sec:appEllReg}, we discuss the elliptic regularity supposed in Assumption (A4).


\section{Preparations}\label{sec.prep}

We define the operator $L_\eps:H^{-1}(\Omega;\R^d)\to H^{-1}(\Omega;\R^d)$ by $L_\eps(g)=v$, where $v\in H^1_0(\Omega;\R^d)$ is the unique weak solution to
\begin{align*}
  -\eps\Delta v + v = g\quad\mbox{in }\Omega, \quad 
  v=0\quad\mbox{on }\pa\Omega,
\end{align*}
where $\eps>0$. The estimate $\|v\|_{H^1(\Omega)}\le C(\eps)\|g\|_{H^{-1}(\Omega)}$ follows from elliptic theory. The operator $L_\eps$ is linear, symmetric, positive, bounded, and self-adjoint. We deduce from spectral theory that there exists the unique square root operator $K_\eps$ on $H^{-1}(\Omega;\R^d)$, fulfilling the same properties. By \cite[Lemma 7]{JVZ24}, it holds for all $g\in H^{-1}(\Omega;\R^d)$ that 
\begin{align}\label{2.Leps}
  \eps\|v\|_{L^2(\Omega)}^2 + \|v\|_{L^2(\Omega)}^2 
  \le \|K_\eps(g)\|_{L^2(\Omega)}^2.
\end{align}

We also need a higher-order regularization. Let $\eta>0$, $m\ge d/2+1$, $m\in \N$, and $X=H^m(\Omega;\R^d)\cap H_0^1(\Omega;\R^d)$. We introduce the operator $L_\eps^\eta:X'\to X'$ by $L_\eps^\eta(g)=v$, where $v\in X$ is the unique solution to
\begin{align*}
  \int_\Omega\bigg(\eta\sum_{|\alpha|=m}D^\alpha v\cdot
  D^\alpha\phi + \eps\na v:\na\phi + v\cdot\phi\bigg)dx
  = \langle g,\phi\rangle
\end{align*}
for all $\phi\in X$, where $\langle\cdot,\cdot\rangle$ is the dual product in $X'\times X$ and ``:'' is the Frobenius matrix product (summation over both indices). Since $m>d/2$, we have the embedding $H^m(\Omega)\hookrightarrow L^\infty(\Omega)$. The operator $L_\eps^\eta$ has the same properties as $L_\eps$, so there exists the unique square root operator $K_\eps^\eta$ on $X'$. Similarly as for $L_\eps$, it holds that
\begin{align}\label{2.KLeta}
  \eta \|v\|_{H^m(\Omega)}^2 
  + \eps\|\na v\|_{L^2(\Omega)}^2 + \|v\|_{L^2(\Omega)}^2
  \le C\|K_\eps^\eta(g)\|_{L^2(\Omega)}^2 \quad\mbox{for }g\in X'.
\end{align}
If Assumption (A4) is satisfied, we conclude additional bounds that are uniform in $\eta$; see Appendix \ref{sec:appEllReg}.


\section{Case $0<\beta<1/d$}\label{sec.beta1} 

In this section, we prove Theorem \ref{thm.ex1}. We split the proof into several steps.

\subsection{Solution to an approximate problem}

Let $T>0$, $N\in\N$, and let $\tau=T/N>0$ be the time step. We introduce the following time-discrete recursive approximate scheme
\begin{align}\label{3.approx1}
  \frac{1}{\tau}\int_\Omega(u_i^k-u_i^{k-1})\phi_i dx
  + \sigma_i\int_\Omega\na u_i^k\cdot\na\phi_i dx
  = \int_\Omega u_i^k v_i^k\cdot\na\phi_i dx, \quad i=1,\ldots,n,
\end{align}
for all $\phi_i\in L^2(0,T;H^1(\Omega))$, where $v_i^k=L_\eps(-\na p_i(u^k))$ and the operator $L_\eps$ is defined in Section \ref{sec.prep}. To apply a fixed-point argument, we first introduce a linear problem. Let $u^{k-1}\in L^2(\Omega;\R^n)$ with $u_i^{k-1}\ge 0$ in $\Omega$ be given, and let $\delta\in[0,1]$, $y=(y_1,\ldots,y_n)\in L^2(\Omega;\R^n)$. We set $\widetilde{v}_i = L_\eps(-\na p_i(y))$ for $i=1,\ldots,n$. Since $2/\beta \geq 2$ and $y_j^\beta\in L^{2/\beta}(\Omega)$, we can apply the result of \cite{Gro89} to conclude that $\widetilde{v}_i\in W^{1,2/\beta}(\Omega;\R^d)\subset L^\infty(\Omega;\R^d)$, where the embedding holds because $\beta<2/d$. We define the approximate linear problem
\begin{align}\label{3.lin1}
  \frac{\delta}{\tau}\int_\Omega&(y_i-u_i^{k-1}) \phi_i dx
  + \sigma_i\int_\Omega\na u_i\cdot\na\phi_i dx
  + \int_\Omega u_i\phi_i dx \\
  &= \delta\int_\Omega(y_i)_+ \widetilde{v}_i\cdot\na\phi_i dx
  + \delta\int_\Omega y_i\phi_i dx \nonumber 
\end{align}
for all $\phi_i\in H^1(\Omega)$, where $(y_i)_+=\max\{0,y_i\}$. The last terms on the left- and right-hand sides cancel if $y_i=u_i$ and $\delta=1$. They are needed to obtain the coercivity of the associated bilinear form (see below). The first integral on the right-hand side exists, since $(y_i)_+\in L^2(\Omega)$ and $\widetilde{v}_i\in L^\infty(\Omega;\R^d)$. We introduce for $\phi_i\in H^1(\Omega)$ the functionals
\begin{align*}
  a(u_i,\phi_i) &= \sigma_i\int_\Omega\na u_i\cdot\na\phi_i dx
  + \int_\Omega u_i\phi_i dx, \\
  F(\phi_i) &= -\frac{\delta}{\tau}\int_\Omega(y_i -u_i^{k-1})\phi_i dx
  + \delta\int_\Omega(y_i)_+\widetilde{v}_i\cdot\na\phi_i dx
  + \delta\int_\Omega y_i\phi_i dx.
\end{align*}
The bilinear form $a$ is coercive and bounded on $H^1(\Omega)$, and the linear form $F$ is bounded on $H^1(\Omega)$. By the Lax--Milgram lemma, there exists a unique solution $u_i\in H^1(\Omega)$ to $a(u_i,\phi_i)=F(\phi_i)$ for $\phi_i\in H^1(\Omega)$.

This defines the fixed-point operator $Q:L^2(\Omega;\R^n)\times[0,1]\to L^2(\Omega;\R^n)$, $Q(y,\delta)=u$. Any fixed point of $Q(\cdot,1)$ with $u_i\ge 0$ in $\Omega$ for $i=1,\ldots,n$ solves \eqref{3.approx1}. It holds that $Q(y_i,0)=0$. Standard arguments show the continuity of $Q$, and the compact embedding $H^1(\Omega)\hookrightarrow L^2(\Omega)$ implies the compactness of $Q$. It remains to find a constant $C>0$ such that $\|u\|_{L^2(\Omega)}\le C$ for all fixed points $u$ of $Q(\cdot,\delta)$ and all $\delta\in(0,1]$. 

Let $u\in H^1(\Omega;\R^n)$ be such a fixed point. First, we use the test function $\phi_i=(u_i)_-=\min\{0,u_i\}$ in \eqref{3.lin1} with $y=u$ and $v_i=L_\eps(-\na p_i(u))$:
\begin{align*}
  \frac{\delta}{\tau}\int_\Omega& (u_i)_-^2 dx
  + \sigma_i\int_\Omega|\na (u_i)_-|^2 dx
  + (1-\delta)\int_\Omega(u_i)_-^2 dx \\
  &= \frac{\delta}{\tau}\int_\Omega u_i^{k-1}(u_i)_- dx
  + \delta\int_\Omega(u_i)_+ v_i\cdot\na(u_i)_- dx\le 0,
\end{align*}
since $u_i^{k-1}(u_i)_-\le 0$ and $(u_i)_+\na(u_i)_-=0$ in $\Omega$. This shows that $u_i\ge 0$ in $\Omega$. Next, we use $\phi_i=1$ as a test function in \eqref{3.lin1} with $y=u$, leading to
\begin{align*}
  \int_\Omega u_i dx = \delta\int_\Omega u_i^{k-1}dx
  - (1-\delta)\int_\Omega u_i dx 
  \le \int_\Omega u_i^{k-1}dx \le \int_\Omega u_i^0 dx,
\end{align*}
where the last step follows by iteration. Hence, $\|u_i\|_{L^1(\Omega)}\le C(u_i^0)$. We infer that $\na u_i^\beta$ is bounded $W^{-1,1/\beta}(\Omega)$ (uniformly in $\delta$) and, by elliptic regularity \cite{Gro89}, $v_i=L_\eps(-\na p_i(u_i))$ is uniformly bounded in $W^{1,1/\beta}(\Omega)$ $\subset L^\infty(\Omega)$ if $\beta<1/d$. This bound, which only depends on $u_i^0$, $\Omega$, and $\eps$, allows us to derive an $H^1(\Omega)$ estimate for $u_i$. Indeed, we use the test function $\phi_i=u_i$ in \eqref{3.lin1} with $y=u$:
\begin{align*}
  \frac{\delta}{\tau}&\int_\Omega(u_i-u_i^{k-1})u_i dx
  + \sigma_i\int_\Omega|\na u_i|^2 dx 
  + (1-\delta)\int_\Omega u_i^2 dx \\
  &= \delta\int_\Omega u_iv_i\cdot\na u_i dx
  \le \frac{\sigma_i}{2}\int_\Omega|\na u_i|^2 dx
  + \frac{\delta^2}{2\sigma_i}\|v_i\|_{L^\infty(\Omega)}^2
  \int_\Omega u_i^2 dx.
\end{align*}
Taking into account the inequality $(u_i-u_i^{k-1})u_i\ge \frac12(u_i^2-(u_i^{k-1})^2)$ and the property $\delta^2\le \delta$, we find that
\begin{align*}
  \frac{\delta}{2\tau}\int_\Omega\big(u_i^2 - (u_i^{k-1})^2\big)dx
  + \frac{\sigma_i}{2}\int_\Omega|\na u_i|^2 dx
  \le \frac{\delta}{2\sigma_i}\|v_i\|_{L^\infty(\Omega)}^2
  \int_\Omega u_i^2 dx.
\end{align*}
Choosing $\tau<\sigma_i/\|v_i\|_{L^\infty(\Omega)}^2$, we obtain a uniform $L^2(\Omega)$ bound for $\na u_i$. Since $u_i$ is uniformly bounded in $L^1(\Omega)$, the Poincar\'e--Wirtinger inequality yields a uniform $H^1(\Omega)$ bound for $u_i$. We conclude from the Leray--Schauder fixed-point theorem that there exists a fixed point of $Q(\cdot,1)$ and hence a weak solution $u_i^k:=u_i$ to \eqref{3.approx1}.  

\subsection{Limit $\tau\to 0$}

We define the piecewise constant in time functions
\begin{align}\label{3.utau}
  u^{(\tau)}(x,t)=u^k(x), \quad v^{(\tau)}(x,t) = v^k(x)
  \quad\mbox{for }x\in\Omega,\ t\in((k-1)\tau,k\tau],
\end{align}
where $k=1,\ldots,N$. For $t=0$, we define $u^{(\tau)}(\cdot,0)=u^0$ in $\Omega$. The time shift is given by $\pi_\tau u^{(\tau)}(x,t)=u^{k-1}(x)$ for $x\in\Omega$ and $t\in((k-1)\tau,k\tau]$. The function $u_i^{(\tau)}$ solves
\begin{align}\label{3.tau}
  \frac{1}{\tau}\int_0^T\int_\Omega(u_i^{(\tau)}
  - \pi_\tau u_i^{(\tau)})\phi_i dxdt
  + \sigma_i\int_0^T\int_\Omega\na u_i^{(\tau)}\cdot\na\phi_i dxdt
  = \int_0^T\int_\Omega u_i^{(\tau)} v_i^{(\tau)}\cdot\na\phi_i dxdt
\end{align}
for all $\phi_i\in L^2(0,T;H^1(\Omega))$, where $v_i^{(\tau)} = L_\eps(-\na p_i(u^{(\tau)}))$. Our estimates imply, after summation over $k=1,\ldots,N$, that 
\begin{align*}
  \|u_i^{(\tau)}\|_{L^\infty(0,T;L^{1}(\Omega))}
  + \|u_i^{(\tau)}\|_{L^2(0,T;H^1(\Omega))} \le C_1, \quad
  \|v_i^{(\tau)}\|_{L^\infty(\Omega_T)} \le C_2(\eps),
\end{align*}
where $C_1$, $C_2(\eps)>0$ are independent of $\tau$. We wish to derive a uniform estimate for the discrete time derivative. Let $\phi_i\in L^2(0,T;H^1(\Omega))$. Then
\begin{align*}
  \frac{1}{\tau}\bigg|&\int_0^T\int_\Omega(u_i^{(\tau)} 
  - \pi_\tau u_i^{(\tau)})\phi_i dxdt\bigg| \\
  &\le \sigma_i\|\na u_i^{(\tau)}\|_{L^2(\Omega_T)}
  \|\na\phi_i\|_{L^2(\Omega_T)}
  + \|u_i^{(\tau)}\|_{L^2(\Omega_T)}
  \|v_i^{(\tau)}\|_{L^\infty(\Omega_T)}\|\na\phi_i\|_{L^2(\Omega_T)} \\
  &\le C(\eps)\|\phi_i\|_{L^2(0,T;H^1(\Omega))}.
\end{align*}
We infer that there exists a constant $C(\eps)>0$ such that for all $\tau>0$,
\begin{align*}
  \tau^{-1}\|u_i^{(\tau)}-\pi_\tau u_i^{(\tau)}\|_{L^2(0,T;H^1(\Omega)')}
  \le C(\eps).
\end{align*}
Thus, we can apply the Aubin--Lions lemma in the version of \cite{DrJu12} to obtain the existence of a subsequence that is not relabeled such that, as $\tau\to 0$,
\begin{align*}
  u_i^{(\tau)}\to u_i\quad \mbox{strongly in }L^2(\Omega_T). 
\end{align*}
Moreover, we have, up to subsequences,
\begin{align*}
  \na u_i^{(\tau)}\rightharpoonup \na u_i &\quad\mbox{weakly in }
  L^2(0,T;H^1(\Omega)), \\
  \pa_t u_i^{(\tau)}\rightharpoonup \pa_t u_i 
  &\quad\mbox{weakly in }L^2(0,T;H^1(\Omega)'), \\
  v_i^{(\tau)}\rightharpoonup^* v_i
  &\quad\mbox{weakly* in }L^\infty(\Omega_T). 
\end{align*}
Thus, the limit $\tau\to 0$ in \eqref{3.tau} shows that $u_i$ solves
\begin{align}\label{3.w}
  \int_0^T\langle \pa_t u_i,\phi_i\rangle dt
  + \sigma_i\int_0^T\int_\Omega\na u_i\cdot\na\phi_i dxdt
  = \int_0^T\int_\Omega u_iv_i\cdot\na\phi_i dxdt.
\end{align}
The bound of $u_i^{(\tau)}$ in $L^2(0,T;H^1(\Omega))\cap H^1(0,T;H^1(\Omega)')\hookrightarrow C^0([0,T];L^2(\Omega))$ implies, for a continuous representative, that $u_i^{(\tau)}\rightharpoonup u_i$ in $C^0([0,T];L^2(\Omega))$ and therefore $u_i(0)=u_i^0$ holds in the sense of $L^2(\Omega)$. 

It remains to identify $v_i$. We can perform the limit $\tau\to 0$ in the weak formulation of $v_i^{(\tau)}=L_\eps(-\na p_i(u^{(\tau)}))$,
\begin{align*}
  -\eps\int_0^T\int_\Omega v_i^{(\tau)}\cdot\Delta\phi dxdt
  + \int_0^T\int_\Omega v_i^{(\tau)}\cdot\phi dxdt
  = \int_0^T\int_\Omega p_i(u^{(\tau)})\diver\phi dxdt
\end{align*}
for $\phi\in C_0^\infty(\Omega_T;\R^d)$. Indeed, we have $p_i(u^{(\tau)})\to p_i(u)$ strongly in $L^{2/\beta}(\Omega_T)$ and hence, in the limit $\tau\to 0$,
\begin{align*}
  -\eps\int_0^T\int_\Omega v_i\cdot\Delta\phi dxdt
  + \int_0^T\int_\Omega v_i\cdot\phi dxdt 
  = \int_0^T\int_\Omega p_i(u)\diver\phi dxdt.
\end{align*}
This gives $v_i=-\na p_i(u)$ in the sense of distributions and finishes the proof of the existence of a weak solution to \eqref{1.u}--\eqref{1.bic}. 

The boundedness of $u_i$ follows as in \cite[Lemma 9]{JVZ24} from an Alikakos iteration. The idea is to use $u_i^{\gamma-1}$ with $\gamma>1$ as a test function in \eqref{3.w} and to apply the Gagliardo--Nirenberg inequality as well as the $L^\infty(\Omega_T)$ bound for $v_i$ to derive the recursive inequality
\begin{align*}
  \frac{d}{dt}\|u_i\|_{L^\gamma(\Omega)}^\gamma
  \le C\gamma^{d/2+1}\|u_i\|_{L^{\gamma/2}(\Omega)}^\gamma.
\end{align*}
This yields for all $\gamma>1$,
\begin{align*}
  \|u_i\|_{L^\gamma(\Omega)}\le C\big(
  \|u_i\|_{L^\infty(0,T;L^1(\Omega))} + \|u_i^0\|_{L^\infty(\Omega)}
  \big) \le C(u_i^0). 
\end{align*}
Since $u_i^\gamma$ may be not an admissible test function for large values of $\gamma$, we need to approximate $u_i^\gamma$. We refer to the proof of \cite[Lemma 9]{JVZ24} for details. The proof in \cite{JVZ24} holds for the time-continuous problem, but the time-discrete equation is estimated in exactly the same way, using the convexity of $z\mapsto z^\gamma$ for $\gamma>1$. This shows that
\begin{align*}
  \|u_i\|_{L^\infty(\Omega_T)} \le C.
\end{align*}
The constant $C>0$ is independent of $\eps$ if the $L^\infty(\Omega_T)$ bound for $v_i$ does not depend on $\eps$. 


\subsection{Uniqueness of solutions}

Since $u_i\in L^\infty(0,T;L^\infty(\Omega))$ and $v_i\in L^\infty(\Omega_T;\R^d)$, the uniqueness proof relies on standard $L^2(\Omega)$ estimations. For completeness, we present the short proof. Let $u$ and $\bar{u}$ be two bounded weak solutions to \eqref{1.u}--\eqref{1.bic} with the same initial data and set $v_i=L_\eps(-\na p_i(u))$, $\bar{v}_i=L_\eps(-\na p_i(\bar{u}))$. The difference $u_i-\bar{u}_i\in L^2(0,T;H^1(\Omega))$ is an admissible test function in the difference of \eqref{3.w}, satisfied for $u_i$ and $\bar{u}_i$, leading to
\begin{align*}
  \frac12\frac{d}{dt}&\int_\Omega(u_i-\bar{u}_i)^2 dx
  + \sigma_i\int_\Omega|\na(u_i-\bar{u}_i)|^2 dx
  = \int_\Omega(u_iv_i-\bar{u}_i\bar{v}_i)\cdot\na(u_i-\bar{u}_i)dx \\
  &\le \frac{\sigma_i}{2}\int_\Omega|\na(u_i-\bar{u}_i)|^2 dx
  + \frac{1}{2\sigma_i}\|u_i\|_{L^\infty(\Omega_T)}^2
  \|v_i-\bar{v}_i\|_{L^2(\Omega_T)}^2 \\
  &\phantom{xx}+ \frac{1}{2\sigma_i}\|u_i-\bar{u}_i\|_{L^2(\Omega_T)}^2
  \|\bar{v}_i\|_{L^\infty(\Omega_T)}^2.
\end{align*}
The linear operator $L^2(\Omega_T)\to L^2(\Omega_T)$, $u_i\mapsto v_i = L_\eps(-\na p_i(u))$ is continuous. This shows that $\|v_i-\bar{v}_i\|_{L^2(\Omega_T)} \le C\|u_i-\bar{u}_i\|_{L^2(\Omega_T)}$ and consequently,
\begin{align*}
  \frac12\frac{d}{dt}\int_\Omega(u_i-\bar{u}_i)^2 dx
  + \frac{\sigma_i}{2}\int_\Omega|\na(u_i-\bar{u}_i)|^2 dx
  \le C\|u_i-\bar{u}_i\|_{L^2(\Omega_T)}^2.
\end{align*}
Gronwall's lemma implies that $\|(u_i-\bar{u}_i)(t)\|_{L^2(\Omega_T)}=0$, hence $u_i(t)=\bar{u}_i(t)$ and $v_i(t)=\bar{v}_i(t)$ in $\Omega$, $t>0$, $i=1,\ldots,n$.


\section{Case $1/2\le\beta<1$ and $d=2$}\label{sec.beta2}

Again, we split the proof of Theorem \ref{thm.ex2} into several steps. The first steps are valid for $\beta\ge 1/2$, but later we need to distinguish between $\beta>1/2$ and $\beta=1/2$. 

\subsection{Solution to an approximate problem}

Similarly to the previous section, we introduce the time-discrete approximate scheme
\begin{align}\label{3.approx2}
  \frac{1}{\tau}\int_\Omega(u_i^k-u_i^{k-1})\phi_i dx
  + \sigma_i\int_\Omega\na u_i^k\cdot\na\phi_i dx
  = \int_\Omega u_i^k v_i^k\cdot\na\phi_i dx, \quad i=1,\ldots,n,
\end{align}
for all $\phi_i\in L^2(0,T;H^1(\Omega))$, where $v_i^k=L_\eps^\eta(-\na p_i(u^k))$ and the operator $L_\eps^\eta$ is defined in Section \ref{sec.prep}. We define the fixed-point operator $Q$ as in the previous section: Let $u^{k-1}\in L^2(\Omega;\R^n)$ with $u_i^{k-1}\ge 0$ in $\Omega$ and let $y\in L^2(\Omega;\R^n)$. Then $p_i(y)\in L^{2/\beta}(\Omega)$ and $\widetilde{v}_i=L_\eps^\eta(-\na p_i(y))\in H^{m}(\Omega;\R^d)\hookrightarrow L^\infty(\Omega;\R^d)$ (since $m>d/2$). Thus, the linear problem \eqref{3.lin1} has a unique solution $u_i\in H^1(\Omega)$. This defines the fixed point operator $Q(y,\delta)=u$. We need to find a uniform bound for all fixed points of $Q(\cdot,\delta)$. 

Let $u\in H^1(\Omega;\R^n)$ be such a fixed point. Then the $L^1(\Omega)$ norm of $u_i$ only depends on $\|u_i^0\|_{L^1(\Omega)}$. Thus, $\na u_i^\beta$ is bounded in $W^{-1,1/\beta}(\Omega;\R^d) \hookrightarrow H^{-m}(\Omega;\R^d)$ (since $m\ge d/2+1$). This shows that $v_i=L_\eps^\eta(-\na p_i(u))$ is bounded in $H^m(\Omega;\R^d)\hookrightarrow L^\infty(\Omega;\R^d)$. The proof of \cite[Lemma 9]{JVZ24} implies that $u_i$ is bounded in $L^\infty(\Omega)$ uniformly in $\delta$. This provides the required uniform $L^2(\Omega)$ bound, and we can apply the Leray--Schauder fixed-point theorem to conclude the existence of a weak solution $u$ to \eqref{3.approx2}.  

\subsection{Approximate entropy inequality}

We aim to derive an approximate entropy inequality. For this, we first observe that the $L^\infty(\Omega)$ bound for $v_i$ implies, as in the previous section, that $\na u_i\in L^2(\Omega;\R^d)$. We claim that also $\na\sqrt{u_i}\in L^2(\Omega;\R^d)$. With the test function $\phi_i=\log(u_i+\mu)\in H^1(\Omega)$ in \eqref{3.approx2} (where $\mu>0$), we infer that
\begin{align}\label{3.aux2}
  0 &= \frac{1}{\tau}\int_\Omega(u_i-u_i^{k-1})\log(u_i+\mu)dx
  + \sigma_i\int_\Omega\na u_i\cdot\na\log(u_i+\mu)dx \\
  &\phantom{xx}
  - \int_\Omega u_i v_i\cdot\na\log(u_i+\mu)dx \nonumber \\
  &\ge \frac{1}{\tau}\int_\Omega\big((u_i+\mu)(\log(u_i+\mu)-1)
   - (u_i^{k-1}+\mu)(\log(u_i^{k-1}+\mu)-1)\big)dx \nonumber \\
  &\phantom{xx}
  + 4\sigma_i\int_\Omega \frac{u_i}{u_i+\mu}|\na\sqrt{u_i}|^2 dx
  - \int_\Omega\frac{u_i}{u_i+\mu}v_i\cdot\na u_i dx,
  \nonumber 
\end{align}
where $C>0$ is independent of $\mu$, taking into account that $z\mapsto z(\log z-1)$ is convex. 
The last integral on the right-hand side is estimated as
\begin{align*}
  \bigg|\int_\Omega\frac{u_i}{u_i+\mu}v_i\cdot\na u_i dx\bigg|
  \le \|v_i\|_{L^2(\Omega)}\|\na u_i\|_{L^2(\Omega)}.
\end{align*}
Thus, we can apply the dominated convergence theorem to perform the limit $\mu\to 0$ in this integral (and also in the first integral on the right-hand side of \eqref{3.aux2}). We conclude from \eqref{3.aux2} in the limit $\mu\to 0$ that
\begin{align*}
  0 \ge \frac{1}{\tau}\int_\Omega\big(u_i(\log u_i-1)
  - u_i^{k-1}(\log u_i^{k-1}-1)\big)dx
  + 4\sigma_i\int_\Omega|\na\sqrt{u_i}|^2 dx 
  - \int_\Omega v_i\cdot\na u_i dx.
\end{align*}
This shows that 
\begin{align*}
  \frac{1}{\tau}&\int_\Omega u_i(\log u_i-1)dx
  + 4\sigma_i\int_\Omega|\na\sqrt{u_i}|^2 dx \\
  &\le \frac{1}{\tau}\int_\Omega u_i^{k-1}(\log u_i^{k-1}-1)dx
  + \|v_i\|_{L^2(\Omega)}\|\na u_i\|_{L^2(\Omega)}.
\end{align*}
We infer that $\sqrt{u_i}\in H^1(\Omega)$. Therefore,
\begin{align*}
  \na u_i^\beta = 2\beta u_i^{\beta-1/2}\na\sqrt{u_i}\in L^2(\Omega).
\end{align*}
At this point, we need $\beta\ge 1/2$. Notice that the previous bounds generally depend on $\eta$. We derive $\eta$-uniform estimates from the approximate entropy inequality that is proved next.

\begin{lemma}\label{lem.aei1}
Let $0<\beta<1$ and let $u^k$ be a weak solution to \eqref{3.approx2}. Then 
\begin{align*}
  \frac{1}{\tau}&\sum_{i=1}^n\int_\Omega
  \big(u_i^k-(u_i^k)^\beta\big)dx
  + \frac{4}{\beta}(1-\beta)\sum_{i=1}^n\int_\Omega\sigma_i
  |\na(u_i^k)^{\beta/2}|^2dx \\
  &+ \alpha(1-\beta) 
  \sum_{i=1}^n\int_\Omega |K_\eps^\eta(\na(u_i^k)^\beta)|^2 dx
  \le \frac{1}{\tau}\sum_{i=1}^n\int_\Omega
  \big(u_i^{k-1}-(u_i^{k-1})^\beta\big)dx.
\end{align*}
\end{lemma}

\begin{proof}
Let $\mu>0$. The idea is to take $\phi_i=1-\beta(u_i^k+\mu)^{\beta-1}$ as a test function in \eqref{3.lin1} with $y=u^k$. The regularization with $\mu$ is necessary since $(u_i^k)^{\beta-1}$ is not defined if $u_i^k=0$ (as $\beta<1$). This gives, after summation over $i=1,\ldots,n$,
\begin{align*}
  0 &= \frac{1}{\tau}\sum_{i=1}^n\int_\Omega(u_i^k-u_i^{k-1})
  (1-\beta(u_i^k+\mu)^{\beta-1})dx
  + \sum_{i=1}^n\sigma_i\int_\Omega\na u_i^k\cdot
  \na(1-\beta(u_i^k+\mu)^{\beta-1})dx \\
  &\phantom{xx}
  - \sum_{i=1}^n\int_\Omega u_i^k v_i^k\cdot
  \na(1-\beta(u_i^k+\mu)^{\beta-1})dx =: I_1 + I_2 + I_3,
\end{align*}
where $v_i^k=L_\eps^\eta(-\na p_i(u^k))$. The convexity of $z\mapsto -(z+\mu)^{\beta}$ implies that 
\begin{align*}
  I_1 \ge \frac{1}{\tau}\sum_{i=1}^n\int_\Omega
  \big(u_i^k - (u_i^k+\mu)^\beta\big)dx
  - \frac{1}{\tau}\sum_{i=1}^n\int_\Omega\big(u_i^{k-1} 
  - (u_i^{k-1}+\mu)^\beta\big)dx.
\end{align*}
The second integral $I_2$ can be formulated as
\begin{align*}
  I_2 = \frac{4}{\beta}(1-\beta)\sum_{i=1}^n\sigma_i\int_\Omega
  \frac{(u_i^k)^{2-\beta}}{(u_i^k+\mu)^{2-\beta}}
  |\na(u_i^k)^{\beta/2}|^2 dx.
\end{align*}
Finally, we write
\begin{align*}
  I_3 &= \beta(\beta-1)\sum_{i=1}^n \int_\Omega
  u_i^k(u_i^k+\mu)^{\beta-2}v_i^k\cdot\na u_i^k dx \\
  &= -(1-\beta)\sum_{i=1}^n\int_\Omega 
  \frac{(u_i^k)^{2-\beta}}{(u_i^k+\mu)^{2-\beta}}
  v_i^k\cdot\na(u_i^k)^{\beta} dx.
\end{align*}
We know that $v_i^k$ and $\na(u_i^k)^\beta$ are bounded in $L^2(\Omega;\R^d)$ uniformly in $\mu$. Thus, we can apply the dominated convergence theorem to perform the limit $\mu\to 0$ in $I_3$. Arguing similarly for the other terms, we arrive in the limit $\mu\to 0$ at
\begin{align}\label{3.aux21}
  0 &= \frac{1}{\tau}\sum_{i=1}^n
  \int_\Omega\big(u_i^k - (u_i^k)^\beta\big)dx
  - \frac{1}{\tau}\int_\Omega\big(u_i^{k-1} 
  - (u_i^{k-1})^\beta\big)dx \\
  &\phantom{xx}+ \frac{4}{\beta}(1-\beta)\sum_{i=1}^n\sigma_i\int_\Omega
  |\na(u_i^k)^{\beta/2}|^2 dx - (1-\beta)\sum_{i=1}^n\int_\Omega 
  v_i^k\cdot\na(u_i^k)^{\beta} dx. \nonumber 
\end{align}
We deduce from the definition $p_i(u^k)=\sum_{j=1}^n a_{ij}u_j^k$ that
\begin{align*}
  -\sum_{i=1}^n&\int_\Omega v_i^k\cdot\na(u_i^k)^{\beta} dx
  = \sum_{i,j=1}^n a_{ij}\int_\Omega L_\eps^\eta(\na (u_j^k)^\beta)
  \cdot\na (u_i^k)^\beta dx \\
  &= \sum_{i,j=1}^n a_{ij}\int_\Omega K_\eps^\eta(\na (u_j^k)^\beta)
  \cdot K_\eps^\eta(\na (u_i^k)^\beta) dx
  \ge \alpha\sum_{i=1}^n\int_\Omega|K_\eps^\eta(\na (u_i^k)^\beta)|^2 dx,
\end{align*}
using the positive definiteness of $(a_{ij})$. Inserting this expression into \eqref{3.aux21} finishes the proof.
\end{proof}

\subsection{Further uniform estimates}

We define the piecewise in time functions $u_i^{(\tau)}$ and $v_i^{(\tau)}$ as in \eqref{3.utau}. Lemma \ref{lem.aei1} implies the following bounds for $i=1,\ldots,n$, being uniform in $\tau$, $\eta$, and $\eps$:
\begin{align*}
  \|u_i^{(\tau)}\|_{L^\infty(0,T;L^1(\Omega))}
  + \|(u_i^{(\tau)})^{\beta/2}\|_{L^2(0,T;H^1(\Omega))}
  + \|K_\eps^\eta(\na (u_i^{(\tau)})^\beta)\|_{L^2(\Omega_T)} \le C.
\end{align*}
It follows from the uniform bound of $\na (u_i^{(\tau)})^\beta$ in $L^\infty(0,T;W^{-1,1/\beta}(\Omega))\hookrightarrow L^\infty(0,T;H^{-m}(\Omega))$ for $m\ge d/2+1$ that $v_i^{(\tau)}=L_\eps^\eta(-\na p_i(u^{(\tau)}))$ is bounded in $L^\infty(0,T;H^m(\Omega))\hookrightarrow L^\infty(0,T;$ $L^\infty(\Omega))$. This bound is uniform in $\tau$ but not in $\eta$ (or $\eps$). Therefore, also the $L^\infty(\Omega_T)$ bound for $u_i^{(\tau)}$ from \cite[Lemma 9]{JVZ24} is not uniform in $\eta$, and we need to argue in a different way than in Section \ref{sec.beta1} to derive further bounds that are uniform in $\eta$. We use the Gagliardo--Nirenberg inequality to prove such bounds. We give a proof for general space dimensions to explain our restriction $d=2$. 

\begin{lemma}\label{lem.eta2}
Let $\beta\ge 1/2$ if $d=2$, $\beta\ge 1-1/d$ if $d\ge 3$, and $q=(\beta d+2)/(d+1)\ge 1$. Then there exist constants $C_1>0$ independent of $\tau$, $\eta$, and $\eps$ and $C_2(\eps)$ independent of $\tau$ and $\eta$, such that
\begin{align*}
  \|u_i^{(\tau)}\|_{L^{\beta+2/d}(\Omega_T)}
  + \|u_i^{(\tau)}\|_{L^q(0,T;W^{1,q}(\Omega))} 
  + \|v_i^{(\tau)}\|_{L^2(\Omega_T)} &\le C_1, \\
  \|v_i^{(\tau)}\|_{L^{(\beta d+2)/(\beta d)}(0,T;L^\gamma(\Omega))}
  + \|v_i^{(\tau)}\|_{L^r(\Omega_T)} &\le C_2(\eps),
\end{align*}
where $r<\infty$ if $\beta=1/2$, $d=2$ and otherwise,
\begin{align*}
  \gamma = \begin{cases}
  \infty &\mbox{if }d=2, \\
  d(\beta d+2)/(\beta d^2-\beta d-2) &\mbox{if }d\ge 3,
  \end{cases}\quad
  r = \begin{cases}
  1+1/\beta+2/(2\beta-1) &\mbox{if }d=2, \\
  (\beta d+2)/(\beta d-1) &\mbox{if }d\ge 3.
  \end{cases}
\end{align*}
\end{lemma}

\begin{proof}
We infer from the Gagliardo--Nirenberg inequality with $\theta = (\beta-2/p)/(\beta+2/d-1)$ that
\begin{align}\label{3.GN1}
  \|(u_i^{(\tau)})^{\beta/2}\|_{L^p(\Omega_T)}^p
  &\le C\int_0^T\|(u_i^{(\tau)})^{\beta/2}\|_{H^1(\Omega)}^{p\theta}
  \|(u_i^{(\tau)})^{\beta/2}\|_{L^{2/\beta}(\Omega)}^{p(1-\theta)}dt \\
  &\le C\|u_i^{(\tau)}\|_{L^\infty(0,T;L^1(\Omega))}^{p\beta(1-\theta)/2}
  \int_0^T\|(u_i^{(\tau)})^{\beta/2}\|_{H^1(\Omega)}^{p\theta}dt. \nonumber 
\end{align}
The right-hand side is bounded uniformly in $(\eps,\eta,\tau)$ if $p\theta=2$, which gives $p=2+4/(\beta d)$. Consequently, $u_i^{(\tau)}$ is uniformly bounded in $L^{p\beta/2}(\Omega_T) = L^{\beta + 2/d}(\Omega_T)$ and $(u_i^{(\tau)})^{1-\beta/2}$ is uniformly bounded in $L^{p\beta/(2-\beta)}(\Omega_T)$. 
Then
\begin{align*}
  \na u_i^{(\tau)} = (2/\beta)(u_i^{(\tau)})^{1-\beta/2}
  \na (u_i^{(\tau)})^{\beta/2}
\end{align*}
is uniformly bounded in $L^q(\Omega_T)$, where $1/q=(2-\beta)/(p\beta) + 1/2$ and hence $q = (\beta d+2)/(d+1)$. It holds $q\ge 1$ if $\beta\ge 1-1/d$. The bound for $(u_i^{(\tau)})^{\beta/2}$ in $L^{p}(\Omega_T)$ implies that $(u_i^{(\tau)})^\beta$ is uniformly bounded in $L^{p/2}(\Omega_T)$ and $\na (u_i^{(\tau)})^\beta$ is bounded in $L^{p/2}(0,T;W^{-1,p/2}(\Omega))$. We conclude from elliptic regularity  (Assumption (A4)), that $v_i^{(\tau)}=L_\eps^\eta(-\na p_i(u^{(\tau)}))$ is bounded in $L^{p/2}(0,T;W^{1,p/2}(\Omega))$ uniformly in $\eta$.

Now, we distinguish the cases $d=2$ and $d\ge 3$. Let first $d=2$. Because of the embedding $W^{1,p/2}(\Omega)$ $\hookrightarrow L^\infty(\Omega)$ if $p/2=1+2/(\beta d)>d$ (which holds true if $\beta<1$ and $d=2$), $v_i^{(\tau)}$ is also uniformly bounded in $L^{p/2}(0,T;L^\infty(\Omega))$. We already noticed that $\na(u_i^{(\tau)})^\beta$ is uniformly bounded in $L^\infty(0,T;W^{-1,1/\beta}(\Omega))$ (this holds in any space dimension). Hence, under Assumption (A4), $v_i^{(\tau)}$ is uniformly bounded in $L^\infty(0,T;W^{1,1/\beta}(\Omega))$. If $\beta=1/2$, $(v_i^{(\tau)})$ is bounded in $L^\infty(0,T;L^r(\Omega))$ for all $r<\infty$. If $\beta>1/2$, we apply the Gagliardo--Nirenberg inequality with $\theta=2/(r(2\beta-1))$:
\begin{align*}
  \|v_i^{(\tau)}\|^r_{L^r(\Omega_T)}
  &\le C\|v_i^{(\tau)}\|_{L^\infty(0,T;W^{1,1/\beta}(\Omega))}^{r\theta}
  \int_0^T\|v_i^{(\tau)}\|_{L^\infty(\Omega)}^{r(1-\theta)}dt.
\end{align*}
The right-hand side is bounded if $r(1-\theta)=p/2$, which is equivalent to $r=(\beta+1)/\beta + 2/(2\beta-1)>4$. 

Next, let $d\ge 3$. The embedding $W^{1,p/2}(\Omega)\hookrightarrow L^\gamma(\Omega)$ with $\gamma = d(d\beta+2)/(\beta d^2-\beta d-2)$ implies that $(v_i^{(\tau)})$ is bounded in $L^{p/2}(0,T;L^\gamma(\Omega))$. Again by the Gagliardo--Nirenberg inequality,
\begin{align*}
  \|v_i^{(\tau)}\|^r_{L^r(\Omega_T)}
  &\le C\|v_i^{(\tau)}\|_{L^\infty(0,T;W^{1,1/\beta}(\Omega))}^{r\theta}
  \int_0^T\|v_i^{(\tau)}\|_{L^\gamma(\Omega)}^{r(1-\theta)}dt.
\end{align*}
This expression is uniformly bounded if $r(1-\theta)=p/2$, which yields after a computation $r = (\beta d+2)/(\beta d-1)$. This ends the proof.
\end{proof}

The estimates of Lemma \ref{lem.eta2} imply a uniform bound for $u_i^{(\tau)}v_i^{(\tau)}$ in some $L^s(\Omega_T)$ with $s\ge 1$ only in the two-dimensional case. Indeed, the $L^{\beta+2/d}(\Omega_T)$ bound for $u_i^{(\tau)}$ and the $L^r(\Omega_T)$ bound for $v_i^{(\tau)}$ show that $(u_i^{(\tau)}v_i^{(\tau)})$ is bounded in $L^s(\Omega_T)$ with $1/s = d/(\beta d+2) + 1/r$, which yields $s>1$ for $d=2$ and $1/2\le\beta<1$. However, if $d\ge 3$, we obtain $s = (\beta d+2)/(\beta d+d-1)$, and this expression is larger than one if and only if $d<3$, contradiction. Thus, let $d=2$ in the following.

We derive a uniform bound for the discrete time derivative. Let $\beta>1/2$. Then $q=\frac23(\beta+1)>1$. As $(u_i^{(\tau)})$ is bounded in $L^{\beta+1}(\Omega_T)$ and $(v_i^{(\tau)})$ is bounded in $L^r(\Omega_T)$, $(u_i^{(\tau)}v_i^{(\tau)})$ is bounded in $L^s(\Omega_T)$ with $s=(\beta+1)r/(\beta+r+1)>q$. We choose $q'=q/(q-1)$ and $\phi_i\in L^{q'}(0,T;W^{1,q'}(\Omega))$. Then
\begin{align*}
  \bigg|\frac{1}{\tau}\int_0^T\int_\Omega(u_i^{(\tau)}-\pi_\tau
  u_i^{(\tau)})\phi_i dxdt\bigg|
  &\le \sigma_i\|\na u_i^{(\tau)}\|_{L^q(\Omega_T)}
  \|\na\phi_i\|_{L^{q'}(\Omega_T)} \\
  &\phantom{xx}+ \|u_i^{(\tau)}v_i^{(\tau)}\|_{L^q(\Omega_T)}
  \|\na\phi_i\|_{L^{q'}(\Omega_T)}
  \le C\|\na\phi_i\|_{L^{q'}(\Omega_T)}.
\end{align*}
This shows that 
\begin{align}\label{3.time2}
  \tau^{-1}\|u_i^{(\tau)}-\pi_\tau u_i^{(\tau)}
  \|_{L^{q}(0,T;W^{1,q'}(\Omega)')} \le C.
\end{align}

\subsection{Limit $(\eta,\tau)\to 0$}

Thanks to the gradient bound for $u_i^{(\tau)}$ from Lemma \ref{lem.eta2} and the time-discrete estimate \eqref{3.time2}, we can apply the Aubin--Lions lemma in the version of \cite{DrJu12} to find a subsequence that is not relabeled such that, as $(\eta,\tau)\to 0$,
\begin{align*}
  u_i^{(\tau)}\to u_i\quad\mbox{strongly in }L^q(\Omega_T),\quad
  q = \tfrac23(\beta+1)>1.
\end{align*}
We conclude from Lemma \ref{lem.eta2} for a subsequence that
\begin{align*}
  v_i^{(\tau)}\rightharpoonup v_i 
  \quad\mbox{weakly in }L^r(\Omega_T), \quad r>4.
\end{align*}
It follows from \eqref{3.time2} that, up to a subsequence,
\begin{align*}
  \pa_t u_i^{(\tau)}\rightharpoonup\pa_t u_i
  \quad\mbox{weakly in }L^{q}(0,T;W^{1,q'}(\Omega)'),
\end{align*}
and the following convergence follows from Lemma \ref{lem.eta2}:
\begin{align*}
  \na u_i^{(\tau)}\rightharpoonup\na u_i 
  \quad\mbox{weakly in }L^q(\Omega_T).
\end{align*}
The uniform bound in $L^{\beta+1}(\Omega_T)$ shows that
\begin{align*}
  u_i^{(\tau)}\to u_i \quad\mbox{strongly in }L^\rho(\Omega_T)
  \mbox{ for }\tfrac32\le\rho<\beta+1.
\end{align*}
In particular, we have $p_i(u^{(\tau)})\to p_i(u)$ strongly in $L^{\rho/\beta}(\Omega_T)$. Since $1/r+1/\rho<1/4+2/3 < 1$, 
\begin{align*}
  u_i^{(\tau)}v_i^{(\tau)}\rightharpoonup u_iv_i
  \quad\mbox{weakly in }L^1(\Omega_T). 
\end{align*}
Thus the limit $(\eta,\tau)\to 0$ in \eqref{3.tau} leads to
\begin{align*}
  \int_0^T\langle\pa_t u_i,\phi_i\rangle dt
  + \sigma_i\int_0^T\int_\Omega\na u_i\cdot\na\phi_i dxdt
  = \int_0^T\int_\Omega u_iv_i\cdot\na\phi_i dxdt
\end{align*}
for all $\phi_i\in L^{q'}(0,T;W^{1,q'}(\Omega))$, where $\langle\cdot,\cdot\rangle$ is the dual product in $W^{1,q}(\Omega)'\times W^{1,q'}(\Omega)$. 

It remains to identify the limit $v_i$. The following arguments hold for $\beta\ge 1/2$. We deduce from \eqref{2.KLeta} that
\begin{align*}
  \sqrt{\eta}\|v_i^{(\tau)}\|_{H^m(\Omega)}
  + \sqrt{\eps}\|\na v_i^{(\tau)}\|_{L^2(\Omega)}
  + \|v_i^{(\tau)}\|_{L^2(\Omega)}
  \le C\|K_\eps(\na p_i(u^{(\tau)}))\|_{L^2(\Omega)}.
\end{align*}
It follows from an integration over time and Lemma \ref{lem.aei1} that
\begin{align*}
  \sqrt{\eta}\|v_i^{(\tau)}\|_{L^{2}(0,T;H^m(\Omega))}
  + \|v_i^{(\tau)}\|_{L^{2}(0,T;H^1(\Omega))} \le C(\eps).
\end{align*}
Therefore, for a subsequence,
\begin{align*}
  \eta v_i^{(\tau)}\to 0
  &\quad\mbox{strongly in }L^{2}(0,T;H^m(\Omega)), \\
  v_i^{(\tau)}\rightharpoonup v_i
  &\quad\mbox{weakly in }L^{2}(0,T;H^1(\Omega)).
\end{align*}
Then, together with the strong convergence $p_i(u^{(\tau)})\to p_i(u)$ in $L^{\rho/\beta}(\Omega_T)$, the limit $(\eta,\tau)\to 0$ in 
\begin{align*}
  \int_0^T\int_\Omega\bigg(\eta\sum_{|\alpha|=m}
  D^\alpha v_i^{(\tau)}\cdot D^\alpha\phi_i
  + \eps\na v_i^{(\tau)}:\na\phi_i + v_i^{(\tau)}\cdot\phi_i\bigg)dxdt
  = \int_0^T\int_\Omega p_i(u^{(\tau)})\diver\phi_i dxdt
\end{align*}
for $\phi_i\in L^{\rho/(\rho-\beta)}(0,T;H^m(\Omega;\R^d))$ yields
\begin{align*}
  \int_0^T\int_\Omega(\eps\na v_i:\na\phi_i + v_i\cdot\phi_i)dxdt
  = \int_0^T\int_\Omega p_i(u)\diver\phi_i dxdt,
\end{align*}
We deduce from the weak convergence $v_i^{(\tau)}\rightharpoonup v_i$ in $L^2(0,T;H^1(\Omega))$ and $v_i^{(\tau)}=0$ on $\pa\Omega$ that $v_i=0$ on $\pa\Omega$ (in the sense of traces). We conclude that $v_i=L_\eps(-\na p_i(u))$. Finally, the property $u_i\in W^{1,q}(0,T;W^{1,q'}(\Omega)') \hookrightarrow C^0([0,T];W^{1,q'}(\Omega)')$ implies that $u_i(0)=u_i^0$ in the sense of $W^{1,q'}(\Omega)'$. This ends the proof of Theorem \ref{thm.ex2}.

\subsection{Case $\beta=1/2$, $d=2$}\label{sec.12}

Lemma \ref{lem.eta2} implies that $(u_i^{(\tau)})$ is bounded in the non-reflexive space $L^1(0,T;W^{1,1}(\Omega))$, and there exists $C>0$ independent of $(\eta,\tau)$ such that
\begin{align*}
  \tau^{-1}\|u_i^{(\tau)}-\pi_\tau u_i^{(\tau)}
  \|_{L^1(0,T;W^{1,\infty}(\Omega)')} \le C.
\end{align*}
Since $L^1(0,T;W^{1,\infty}(\Omega)')$ is not reflexive, we cannot extract a weakly convergent subsequence. However, we can still apply the Aubin--Lions lemma to find that (for a subsequence)
\begin{align*}
  u_i^{(\tau)}\to u_i\quad\mbox{strongly in }L^1(\Omega_T)
  \quad\mbox{as }(\eta,\tau)\to 0,
\end{align*}
and this convergence also holds in $L^\rho(\Omega_T)$ with $\rho<\beta+1=3/2$. In particular, $p_i(u^{(\tau)})\to p_i(u)$ strongly in $L^{2\rho}(\Omega_T)$. We write the weak formulation \eqref{3.tau} as
\begin{align*}
  \frac{1}{\tau}\int_{\tau}^{T-\tau}&\int_\Omega u_i^{(\tau)}
  (\phi_i - \pi_{-\tau}\phi_i) dxdt
  + \frac{1}{\tau}\int_0^\tau\int_\Omega
  u_i^{(\tau)}(t)\phi_i(t)dxdt \\
  &- \sigma_i\int_0^T\int_\Omega u_i^{(\tau)}\Delta\phi_i dxdt
  = \int_0^T\int_\Omega u_i^{(\tau)} v_i^{(\tau)}\cdot\na\phi_i dxdt
\end{align*}
for piecewise constant in time functions $\phi_i$ with values in $W^{1,\infty}(\Omega)\cap H^2(\Omega)$ satisfying $\phi_i(t)=0$ for $t>T-\tau$ and $\na\phi_i\cdot\nu=0$ on $\pa\Omega$. The limit $(\eta,\tau)\to 0$ yields
\begin{align*}
  -\int_0^T\int_\Omega u_i\pa_t\phi_i dxdt 
  - \int_\Omega u_i^0\phi_i(0)dx
  - \sigma_i\int_0^T\int_\Omega u_i\Delta\phi_i dxdt
  = \int_0^T\int_\Omega u_i v_i\cdot\na\phi_i dxdt
\end{align*}
for functions $\phi_i\in C_0^\infty(\Omega\times[0,T))$ with $\na\phi_i\cdot\nu=0$ on $\pa\Omega$. The proof is finished.


\section{Case $2d/(d+2)<\beta<2$}\label{sec.beta3}

We turn to the proof of Theorem \ref{thm.ex3}. As in Section \ref{sec.beta2}, we define the time-discrete approximate scheme \eqref{3.approx2}, where $v_i^k=L_\eps^\eta(-\na p_i(u^k))$. The existence of a weak solution $u^k$ to \eqref{3.approx2} is shown as in Section \ref{sec.beta2}, using the fact that for given $u^k\in L^2(\Omega_T;\R^n)$, we have $(u_j^k)^\beta\in L^{2/\beta}(\Omega)\hookrightarrow H^m(\Omega)'$ (here we need that $2/\beta\ge 1$ and $m>d/2$) and hence $v_i^k=L_\eps^\eta(-\na p_i(u^k))\in H^m(\Omega)\hookrightarrow L^\infty(\Omega)$. The approximate entropy inequality can be proved as in Lemma \ref{lem.aei1} with the exception that we use the test function $\phi_i=\beta (u_i^k+\mu)^{\beta-1}-1$ with $\mu>0$. The regularization with $\mu>0$ is still needed to obtain $\phi_i\in H^1(\Omega)$. The result reads as follows.

\begin{lemma}
Let $u^k$ be a weak solution to \eqref{3.approx2}. Then 
\begin{align*}
  \frac{1}{\tau}&\sum_{i=1}^n\int_\Omega
  \big((u_i^k)^\beta-u_i^k\big)dx
  + \frac{4}{\beta}(\beta-1)\sum_{i=1}^n\int_\Omega\sigma_i
  |\na(u_i^k)^{\beta/2}|^2dx \\
  &+ \alpha(\beta-1)\sum_{i=1}^n\int_\Omega 
  |K_\eps^\eta(\na(u_i^k)^\beta)|^2 dx
  \le \frac{1}{\tau}\sum_{i=1}^n\int_\Omega
  \big((u_i^{k-1})^\beta-u_i^{k-1}\big)dx.
\end{align*}
\end{lemma}
Thus, with the piecewise constant in time functions $u_i^{(\tau)}$ and $v_i^{(\tau)}$, defined in \eqref{3.utau}, we obtain the following bounds:
\begin{equation}\label{3.beta2}
\begin{aligned}
  \|u_i^{(\tau)}\|_{L^\infty(0,T;L^\beta(\Omega))}
  + \|(u_i^{(\tau)})^{\beta/2}\|_{L^2(0,T;H^1(\Omega))} &\le C, \\
  \|K_\eps^\eta(\na (u_i^{(\tau)})^\beta)\|_{L^2(\Omega_T)} 
  + \|v_i^{(\tau)}\|_{L^2(\Omega_T)} &\le C, 
\end{aligned}
\end{equation}
where $C>0$ is independent of $(\eps,\eta,\tau)$, and the last bound follows from \eqref{2.KLeta}. Since $\beta>1$, the subsequent estimates differ from those in Lemma \ref{lem.eta2}.

\begin{lemma}\label{lem.eta3}
Let $q=\beta(d+2)/(\beta+d)>1$. Then there exists a constant $C>0$ independent of $\tau$ and $\eta$ such that
\begin{align*}
  \|u_i^{(\tau)}\|_{L^{\beta(d+2)/d}(\Omega_T)}
  + \|u_i^{(\tau)}\|_{L^q(0,T;W^{1,q}(\Omega))}
  \le C.
\end{align*}
\end{lemma}

\begin{proof}
We apply the Gagliardo--Nirenberg inequality with $p=2+4/d$ and $\theta=d/2-d/p$:
\begin{align}\label{3.GN}
  \|(u_i^{(\tau)})^{\beta/2}\|_{L^p(\Omega_T)}^p
  &\le C\int_0^T\|(u_i^{(\tau)})^{\beta/2}\|_{H^1(\Omega)}^{p\theta}
  \|(u_i^{(\tau)})^{\beta/2}\|_{L^2(\Omega)}^{p(1-\theta)}dt \\
  &\le C\|u_i^{(\tau)}\|_{L^\infty(0,T;
  L^\beta(\Omega))}^{\beta p(1-\theta)/2}
  \int_0^T\|(u_i^{(\tau)})^{\beta/2}\|_{H^1(\Omega)}^2 dt, \nonumber 
\end{align}
and the right-hand side is bounded in view of \eqref{3.beta2}. This shows that $(u_i^{(\tau)})$ is bounded in $L^{\beta p/2}(\Omega_T) = L^{\beta(d+2)/d}(\Omega_T)$. We infer that
\begin{align*}
  \na u_i^{(\tau)} = (2/\beta)(u_i^{(\tau)})^{1-\beta/2}
  \na(u_i^{(\tau)})^{\beta/2}
\end{align*} 
is uniformly bounded in $L^q(\Omega_T)$, where $1/q = (2-\beta)/(\beta p) + 1/2$ or $q = \beta(d+2)/(d+\beta)>1$. We infer uniform bounds for $(u_i^{(\tau)})^\beta$ in $L^{1+2/d}(\Omega_T)$ and for $\na(u_i^{(\tau)})^\beta$ in $L^{1+2/d}(0,T;$ $W^{-1,1+2/d}(\Omega))$. 
\end{proof}

We claim that our estimates are sufficient to find that $(u_i^{(\tau)}v_i^{(\tau)})$ is bounded in some $L^s(\Omega_T)$ with $s>1$. In view of the bounds for $v_i^{(\tau)}$ in $L^2(\Omega_T)$ and for $u_i^{(\tau)}$ in $L^{\beta(d+2)/d}(\Omega_T)$, this is the case if $d/(\beta(d+2)) + 1/2 = 1/s < 1$, which is equivalent to $\beta>2d/(d+2)$. This is true if $d=2$, but it gives the restriction $2d/(d+2)<\beta<2$ for $d\ge 3$. 

We derive a uniform bound for the discrete time derivative. Let $\phi_i\in L^{s'}(0,T;W^{1,s'}(\Omega))$, where $s'=s/(s-1)$ and $s=2\beta(d+2)/(\beta d + 2\beta + 2d)>1$. Then, since $s<q$,
\begin{align*}
  \bigg|\frac{1}{\tau}\int_0^T\int_\Omega(u_i^{(\tau)}
  - \pi_\tau u_i^{(\tau)})\phi_i dxdt\bigg|
  &\le \sigma_i\|\na u_i^{(\tau)}\|_{L^s(\Omega_T)}
  \|\na\phi_i\|_{L^{s'}(\Omega_T)} \\
  &\phantom{xx}+ \|u_i^{(\tau)}v_i^{(\tau)}\|_{L^s(\Omega_T)}
  \|\na\phi_i\|_{L^{s'}(\Omega_T)} 
  \le C\|\na\phi_i\|_{L^{s'}(\Omega_T)},
\end{align*} 
and we conclude that $\tau^{-1}(u_i^{(\tau)}-\pi_\tau u_i^{(\tau)})$ is uniformly bounded in $L^s(0,T;W^{1,s'}(\Omega)')$. Thus, we can apply the Aubin--Lions lemma in the version of \cite{DrJu12} to find a subsequence (not relabeled) such that, as $(\eta,\tau)\to 0$,
\begin{align*}
  u_i^{(\tau)} \to u_i \quad\mbox{strongly in }L^s(\Omega_T),
  \quad s>1.
\end{align*}
The uniform bound for $u_i^{(\tau)}$ in $L^{\beta(d+2)/d}(\Omega_T)$ implies that this convergence also holds in $L^\rho(\Omega_T)$ for all $\rho<\beta(d+2)/d$. Furthermore, up to a subsequence,
\begin{align*}
  \na u_i^{(\tau)}\rightharpoonup\na u_i
  &\quad\mbox{weakly in }L^q(\Omega_T), \\
  \pa_t u_i^{(\tau)}\rightharpoonup\pa_t u_i
  &\quad\mbox{weakly in }L^s(0,T;W^{1,s'}(\Omega)'), \\
  v_i^{(\tau)}\rightharpoonup v_i
  &\quad\mbox{weakly in }L^2(\Omega_T). 
\end{align*}
Since it holds $\rho>2$ (a consequence of $\beta>2d/(d+2)$), the strong convergence of $u_i^{(\tau)}$ in $L^\rho(\Omega_T)$ and the weak convergence of $v_i^{(\tau)}$ in $L^2(\Omega_T)$ imply that 
\begin{align*}
  u_i^{(\tau)}v_i^{(\tau)}\rightharpoonup u_iv_i
  \quad\mbox{weakly in }L^1(\Omega_T),
\end{align*}
and thanks to the uniform bound of $u_i^{(\tau)}v_i^{(\tau)}$ in $L^s(\Omega_T)$, this convergence also holds in $L^s(\Omega_T)$, where $s>1$. The limit $v_i$ and the initial data are identified as in Section \ref{sec.beta2}. 


\section{Case $\beta\ge 2$}\label{sec.beta4}

In contrast to the previous sections, we need a truncation in the approximate scheme to prove Theorem \ref{thm.ex4}. Indeed, choosing $y\in L^2(\Omega_T;\R^n)$, the term $y_j^\beta$ may be not integrable if $\beta$ is large, and $v_i=L_\eps^\eta(-\na p_i(y))$ may be not defined. For this reason, we introduce some cutoff functions. Let $N>1$ and set $(z)_+^N:=\max\{0,\min\{N,z\}\}$ for $z\in\R$. We approximate $z^{\beta-1}$ by
\begin{align*}
  S_N^{\beta-1}(z) := (\beta-1)\int_0^z[(s)_+^N]^{\beta-2}ds
  = \begin{cases}
  0 &\mbox{if }z\le 0, \\
  z^{\beta-1} &\mbox{if }0<z\le N, \\
  N^{\beta-1} + (\beta-1) N^{\beta-2}(z-N) &\mbox{if }z>N.
  \end{cases}
\end{align*}
For differentiable functions $f$, the chain rule $\na S_N^{\beta-1}(f)=(\beta-1)[(f)_+^N]^{\beta-2}\na f$ is satisfied. Furthermore, we approximate $z^{\beta}$ by
\begin{align*}
  R_N^{\beta}(z) &:= \beta\int_0^z S_N^{\beta-1}(s)ds \\
  &= \begin{cases}
  0 &\mbox{if }z\le 0, \\
  z^{\beta} &\mbox{if }0<z\le N, \\
  N^\beta + \beta N^{\beta-1}(z-N) + \frac12(\beta-1)\beta
  N^{\beta-2}(z-N)^2 &\mbox{if }z>N.
  \end{cases} \nonumber 
\end{align*}
This definition guarantees the chain rule $\na R_N^{\beta}(f)=\beta S_N^{\beta-1}(f)\na f$.

\subsection{Solution of an approximate problem}

We introduce the following approximate scheme:
\begin{align}\label{3.approx4}
  \frac{1}{\tau}\int_\Omega(u_i^k-u_i^{k-1})\phi_i dx
  + \sigma_i\int_\Omega\na u_i^k\cdot\na\phi_i dx
  = \int_\Omega(u_i^k)_+^N v_i^k\cdot\na\phi_i dx
\end{align}
for all $\phi_i\in H^1(\Omega)$, where 
\begin{align*}
  v_i^k = L_\eps^\eta\bigg(-\sum_{j=1}^n a_{ij}\na S_N^\beta(u_j^k)
  \bigg).
\end{align*}
Let $u^{k-1}\in L^2(\Omega)$ with $u_i^{k-1}\ge 0$ in $\Omega$, $\delta\in[0,1]$, and $y\in L^2(\Omega)$ be given. We set
\begin{align*}
  \widetilde{v}_i = L_\eps^\eta\bigg(-\sum_{j=1}^n a_{ij}
  \na S_N^\beta(y_j)\bigg).
\end{align*}
Since $S_N^\beta(y_j)$ grows at most linearly, we have $S_N^\beta(y_j)\in L^2(\Omega)$. This yields $\widetilde{v}_i\in H^m(\Omega)\hookrightarrow L^\infty(\Omega)$ (since $m>d/2$). The linear problem reads as 
\begin{align}\label{3.lin4}
  \frac{\delta}{\tau}\int_\Omega(y_i-u_i^{k-1})\phi_i dx
  &+ \sigma_i\int_\Omega\na u_i^k\cdot\na\phi_i dx
  + \int_\Omega u_i^k\phi_i dx \\
  &= \delta\int_\Omega(y_i)_+^N\widetilde{v}_i\cdot\na\phi_i dx
  + \delta\int_\Omega y_i\phi_i dx \nonumber 
\end{align}
for all $\phi_i\in H^1(\Omega)$. It follows as in Section \ref{sec.beta1} that there exists a unique solution $u^k\in H^1(\Omega)$ to \eqref{3.lin4}. This defines the fixed-point operator $Q:L^2(\Omega;\R^n)\times[0,1]\to L^2(\Omega;\R^n)$, $Q(y,\delta)=u^k$. Like in the previous sections, the main task is to find an estimate for $u^k$ uniform in $\delta$. To this end, let $u^k$ be such a fixed point. The test function $\phi_i=(u_i)_-$ in \eqref{3.lin4} yields $u_i^k\ge 0$ in $\Omega$. We derive uniform bounds from the following approximate entropy inequality.

\begin{lemma}\label{lem.aei4}
Let $u^k$ be a fixed point of $Q(\cdot,\delta)$ with $0<\delta\le 1$. Then there exists a constant $C>0$, only depending on $u^0$ and $\Omega$, such that 
\begin{align*}
  \frac{\delta}{\tau}&\sum_{i=1}^n\int_\Omega
  \big(R_N^\beta(u_i^k)-u_i^k\big)dx
  + \frac{4}{\beta}(\beta-1)\sum_{i=1}^n\sigma_i\int_\Omega
  |\na S_N^{\beta/2}(u_i^k)|^2dx \\
  &+ \alpha(\beta-1)\delta\sum_{i=1}^n\int_\Omega 
  |K_\eps^\eta(\na S_N^\beta(u_i^k))|^2 dx
  \le \frac{1}{\tau}\int_\Omega
  \big(R_N^\beta(u_i^{k-1})-u_i^{k-1}\big)dx
  + (1-\delta)C.
\end{align*}
\end{lemma}

\begin{proof}
We use the test function $\phi_i = \beta S_N^{\beta-1}(u_i^k)-1$ in \eqref{3.lin4}, which is admissible since $\na S_N^{\beta-1}(u_i^k) = [(u_i^k)_+^N]^{\beta-2}\na u_i^k\in L^2(\Omega)$. At this point, we use the condition $\beta\ge 2$. We obtain, after summation over $i=1,\ldots,n$,
\begin{align*}
  0 &= \frac{\delta}{\tau}\sum_{i=1}^n\int_\Omega
  (u_i^k-u_i^{k-1})(\beta S_N^{\beta-1}(u_i^k)-1)dx
  + \sum_{i=1}^n\sigma_i\int_\Omega
  \na u_i^k\cdot\na(\beta S_N^{\beta-1}(u_i^k)-1)dx \\
  &\phantom{xx}+ (1-\delta)\sum_{i=1}^n\int_\Omega
  u_i^k(\beta S_N^{\beta-1}(u_i^k)-1)dx
  - \delta\sum_{i=1}^n\int_\Omega (u_i^k)_+^N v_i^k\cdot\na 
  (\beta S_N^{\beta-1}(u_i^k)-1)dx \\
  &=: J_1+\cdots+J_4.
\end{align*}
Since $R_N^\beta$ is convex and $(R_N^\beta)'(z)=\beta S_N^{\beta-1}(z)$, we find that
\begin{align*}
  J_1 \ge \frac{\delta}{\tau}\sum_{i=1}^n\int_\Omega
  \big(R_N^\beta(u_i^k)-u_i^k\big)dx
  - \frac{\delta}{\tau}\sum_{i=1}^n\int_\Omega
  \big(R_N^\beta(u_i^{k-1})-u_i^{k-1}\big)dx.
\end{align*}
Taking into account the chain rule, the second term $J_2$ can be formulated as
\begin{align*}
  J_2 = \beta(\beta-1)\sum_{i=1}^n\sigma_i\int_\Omega
  [(u_i^k)_+^N]^{\beta-2}|\na u_i^k|^2 dx
  = \frac{4}{\beta}(\beta-1)\sum_{i=1}^n\sigma_i\int_\Omega
  |\na S_N^{\beta/2}(u_i^k)|^2 dx.
\end{align*}
In view of $S_N^{\beta-1}(u_i^k)\ge 0$, the third term $J_3/(1-\delta)$ is bounded from below by a constant that depends only on $\Omega$. Finally, we compute
\begin{align*}
  J_4 &= -\beta(\beta-1)\delta 
  \sum_{i=1}^n\int_\Omega [(u_i^k)_+^N]^{\beta-1}
  L_\eps^\eta\bigg(-\sum_{j=1}^n a_{ij}\na S_N^\beta(u_j^k)
  \bigg)\cdot\na u_i^k dx \\
  &= (\beta-1)\delta\sum_{i,j=1}^n a_{ij}\int_\Omega 
  L_\eps^\eta(\na S_N^\beta(u_j^k))\cdot\na S_N^\beta(u_i^k) dx \\
  &= (\beta-1)\delta\sum_{i,j=1}^n a_{ij}\int_\Omega 
  K_\eps^\eta(\na S_N^\beta(u_j^k))\cdot
  K_\eps^\eta(\na S_N^\beta(u_i^k))dx \\
  &\ge \alpha(\beta-1)\delta \sum_{i=1}^n \int_\Omega 
  |K_\eps^\eta(\na S_N^\beta(u_i^k))|^2 dx,
\end{align*}
where we used the positive definiteness of $(a_{ij})$ in the last step. We summarize the previous estimations:
\begin{align*}
  0 &= \frac{\delta}{\tau}\sum_{i=1}^n\int_\Omega
  \big(R_N^\beta(u_i^k)-u_i^k\big)dx
  - \frac{\delta}{\tau}\sum_{i=1}^n\int_\Omega
  \big(R_N^\beta(u_i^{k-1})-u_i^{k-1}\big)dx - (1-\delta)C(\Omega) \\
  &\phantom{xx}+ \frac{4}{\beta}(\beta-1)
  \sum_{i=1}^n\sigma_i\int_\Omega|\na S_N^{\beta/2}(u_i^k)|^2 dx
  + \alpha(\beta-1)\delta\sum_{i=1}^n\int_\Omega 
  |K_\eps^\eta(\na S_N^\beta(u_i^k))|^2 dx.
\end{align*} 
This ends the proof.
\end{proof}

The approximate entropy inequality from Lemma \ref{lem.aei4} yields a uniform bound for $u_i^k$ in $L^2(\Omega)$ (since $R_N^\beta(z)\ge z^2$ for $\beta\ge 2$), which provides the required bound for the fixed-point operator. We conclude from the Leray--Schauder fixed-point theorem that there exists a weak solution $u_i^k$ to \eqref{3.approx4}. 

\subsection{Limit $N\to\infty$}

 Since $v_i^k$ is bounded in $H^m(\Omega)\hookrightarrow L^\infty(\Omega)$ uniformly in $N$ (but not in $\eta$), we can apply the Alikakos iteration technique and prove similarly as in \cite[Lemma 9]{JVZ24} that there exists $C(\eta)>0$ such that $\|u_i^k\|_{L^\infty(\Omega)}\le C(\eta)$. Using $\phi_i=u_i^k$ as a test function in \eqref{3.approx4} then yields after standard estimations a bound for $u_i^{(N)}:=u_i^k$ in $H^1(\Omega)$ uniform in $N$ (but not in $\eta$). By compactness, there exists a subsequence (not relabeled) such that
\begin{align*}
  u_i^{(N)}\to u_i \quad\mbox{strongly in }L^2(\Omega)
  \mbox{ as }N\to\infty.
\end{align*}
Taking into account that $S_N^\beta(z)\to z^\beta$ for $z\ge 0$ pointwise as $N\to\infty$ and the uniform $L^\infty(\Omega)$ bound for $u_i^{(N)}$, we obtain $S_N^\beta(u_i^{(N)})\to u_i^\beta$ strongly in $L^p(\Omega)$ for any $p<\infty$. Moreover, we have
\begin{align*}
  \na u_i^{(N)}\rightharpoonup\na u_i
  &\quad\mbox{weakly in }L^2(\Omega), \\
  \na S_N^\beta(u_i^{(N)})\rightharpoonup \na u_i^\beta, \quad
  \na S_N^{\beta/2}(u_i^{(N)})\rightharpoonup \na u_i^{\beta/2}
  &\quad\mbox{weakly in }L^2(\Omega), \\
  L_\eps^\eta\bigg(-\sum_{j=1}^n a_{ij}
  \na S_N^\beta(u_j^{(N)})\bigg)\rightharpoonup 
  L_\eps^\eta(-\na p_i(u)) =: v_i
  &\quad\mbox{weakly in }H^m(\Omega). 
\end{align*}
Thus, we can pass to the limit $N\to\infty$ in \eqref{3.approx4} to find that $u_i^k:=u_i$ solves
\begin{align*}
  \frac{1}{\tau}\int_\Omega(u_i^k-u_i^{k-1})\phi_i dx
  + \sigma_i\int_\Omega\na u_i^k\cdot\na\phi_i dx
  = \int_\Omega u_i^k v_i^k\cdot\na\phi_i dx
\end{align*}
for all $\phi_i\in H^1(\Omega)$, $i=1,\ldots,n$. We can also pass to the limit $N\to\infty$ in the approximate entropy inequality, taking into account the weak lower semicontinuity of the norms, to infer that
\begin{align}\label{3.aei4}
  \frac{1}{\tau}&\sum_{i=1}^n\int_\Omega((u_i^k)^\beta-u_i^k)dx
  + \frac{4}{\beta}(\beta-1)\sum_{i=1}^n\sigma_i\int_\Omega 
  |\na (u_i^k)^{\beta/2}|^2 dx \\
  &+ \alpha(\beta-1)\sum_{i=1}^n\int_\Omega
  |K_\eps^\eta(\na (u_i^k)^\beta)|^2 dx
  \le \frac{1}{\tau}\int_\Omega((u_i^{k-1})^\beta-u_i^{k-1})dx. \nonumber 
\end{align}

\subsection{Limit $(\eta,\tau)\to 0$}

Summing \eqref{3.aei4} over $k=1,\ldots,N$ and using notation \eqref{3.utau}, we deduce the following estimates uniform in $(\eta,\tau)$ and $\eps$:
\begin{align*}
  \|u_i^{(\tau)}\|_{L^\infty(0,T;L^\beta(\Omega))}
  + \|(u_i^{(\tau)})^{\beta/2}\|_{L^2(0,T;H^1(\Omega))} 
  + \|v_i^{(\tau)}\|_{L^2(\Omega_T)} \le C.
\end{align*}
We infer from the Gagliardo--Nirenberg inequality \eqref{3.GN} that $(u_i^{(\tau)})$ is bounded in the space $L^{\beta(d+2)/d}(\Omega_T)$. This shows as in Section \ref{sec.beta3} that $(u_i^{(\tau)}v_i^{(\tau)})$ is bounded in $L^s(\Omega_T)$ with $1/s = d/(\beta(d+2)) + 1/2$, leading to $s = 2\beta(d+2)/(2d+\beta(d+2))>1$. 

Let $\phi_i\in L^{s'}(0,T;W^{2,s'}(\Omega))$ with $s'=s/(s-1)$, satisfying $\na\phi_i\cdot\nu=0$ on $\pa\Omega$. We estimate the discrete time derivative:
\begin{align*}
  \bigg|\frac{1}{\tau}\int_0^T\int_\Omega
  &(u_i^{(\tau)}-\pi_\tau u_i^{(\tau)})\phi_i dx\bigg|
  \le \sigma_i\|u_i^{(\tau)}\|_{L^2(\Omega_T)}
  \|\Delta\phi_i\|_{L^2(\Omega_T)} \\
  &+ \|u_i^{(\tau)}v_i^{(\tau)}\|_{L^s(\Omega_T)}
  \|\na\phi_i\|_{L^{s'}(\Omega_T)}
  \le C\|\phi_i\|_{L^{s'}(0,T;W^{2,s'}(\Omega))}.
\end{align*}
Hence, $\tau^{-1}(u_i^{(\tau)}-\pi_\tau u_i^{(\tau)})$ is uniformly bounded in $L^s(0,T;W^{2,s'}(\Omega)')$. We deduce from the Aubin--Lions lemma in the version of \cite{CJL14} that, up to a subsequence, as $(\eta,\tau)\to 0$,
\begin{align*}
  u_i^{(\tau)}\to u_i \quad\mbox{strongly in }L^\beta(\Omega_T).
\end{align*}
In view of the uniform $L^{\beta(d+2)/2}(\Omega_T)$ bound for $u_i^{(\tau)}$, this convergence also holds in $L^\rho(\Omega_T)$ for any $2<\rho<\beta(d+2)/d$. Furthermore, it holds that
\begin{align*}
  \pa_t u_i^{(\tau)}\rightharpoonup\pa_t u_i
  &\quad\mbox{weakly in }L^{s}(0,T;W^{2,s'}(\Omega)'), \\
  v_i^{(\tau)}\rightharpoonup v_i
  &\quad\mbox{weakly in }L^2(\Omega_T). 
\end{align*}
In particular, we have $u_i^{(\tau)}v_i^{(\tau)}\rightharpoonup u_iv_i$ weakly in $L^s(\Omega_T)$ for some $s>1$. However, we do not obtain a bound on $\na u_i^{(\tau)}$, except if $\beta=2$, and we cannot expect any weak convergence $\na u_i^{(\tau)}\rightharpoonup\na u_i$. The limit $(\eta,\tau)\to 0$ in
\begin{align*}
  \frac{1}{\tau}\int_0^T\int_\Omega(u_i^{(\tau)}-\pi_\tau u_i^{(\tau)})
  \phi_i dxdt - \sigma_i\int_0^T\int_\Omega u_i^{(\tau)}\Delta\phi_i dxdt
  = \int_0^T\int_\Omega u_i^{(\tau)}v_i^{(\tau)}\cdot\na\phi_i dxdt
\end{align*}
for any $\phi_i\in L^\infty(0,T;W^{2,\infty}(\Omega))$ satisfying $\na\phi_i\cdot\nu=0$ on $\pa\Omega$ yields
\begin{align*}
  \int_0^T\langle\pa_t u_i,\phi_i\rangle dt
  - \sigma_i\int_0^T\int_\Omega u_i\Delta\phi_i dxdt
  = \int_0^T\int_\Omega u_iv_i\cdot\na\phi_i dxdt.
\end{align*}
Finally, the limit $v_i$ is identified exactly as in the previous section, since $u_i^{(\tau)}$ converges strongly in $L^\beta(\Omega_T)$ (this implies that $p_i(u^{(\tau)})\to p_i(u)$ strongly in $L^1(\Omega_T)$), and the initial data is satisfied in the sense of $W^{2,s'}(\Omega)'$. 


\section{Localization limit}\label{sec.loc}

The proofs of the previous sections show that the estimates from the entropy inequality are independent of $\eps$. Let $(u^{(\eps)},v^{(\eps)})$ be a weak solution to \eqref{1.u}--\eqref{1.bic}. Then there exists a constant $C>0$, independent of $\eps$ such that
\begin{align*}
  \|u_i^{(\eps)}\|_{L^\infty(0,T;L^{\max\{1,\beta\}}(\Omega))}
  + \|(u_i^{(\eps)})^{\beta/2}\|_{L^2(0,T;H^1(\Omega))}
  + \|v_i^{(\eps)}\|_{L^2(\Omega_T)} \le C.
\end{align*}
In case $\beta<1$, the Gagliardo--Nirenberg inequality \eqref{3.GN1} shows that $(u_i^{(\eps)})$ is bounded in $L^{\beta+2/d}(\Omega_T)$, while in case $\beta>1$, Lemma \ref{lem.eta3} gives a uniform bound for $u_i^{(\eps)}$ in $L^{\beta(d+2)/d}(\Omega_T)$. 

\subsection{Case $0<\beta<1/d$, $d=1$}

In this case, we find that $(u_i^{(\eps)})$ is bounded in $L^{\beta+2}(\Omega_T)$. Since $(v_i^{(\eps)})$ is bounded in $L^2(\Omega_T)$, the product $(u_i^{(\eps)}v_i^{(\eps)})$ is bounded in $L^s(\Omega_T)$ for some $s>1$. Furthermore, $(u_i^{(\eps)})^{(2-\beta)/2}$ is uniformly bounded in $L^{2(\beta+2)/(2-\beta)}(\Omega_T)$ and consequently,
\begin{align*}
  \na u_i^{(\eps)} = (2/\beta)(u_i^{(\eps)})^{(2-\beta)/2}
  \na(u_i^{(\eps)})^{\beta/2}
\end{align*}
is uniformly bounded in $L^r(\Omega_T)$ with $1/r = (2-\beta)/(2(\beta+2)) + 1/2 = 2/(\beta+2)<1$. We conclude that $(u_i^{(\eps)})$ is bounded in $L^r(0,T;W^{1,r}(\Omega))$, and  $(\pa_t u_i^{(\eps)})$ is bounded in $L^s(0,T;W^{1,s'}(\Omega)')$, where $s'=s/(s-1)$. 

The Aubin--Lions lemma \cite{Sim87} yields the existence of a subsequence (not relabeled) such that, as $\eps\to 0$,
\begin{align*}
  u_i^{(\eps)}\to u_i\quad\mbox{strongly in }L^r(\Omega_T),
\end{align*} 
and the uniform bound in $L^{\beta+2}(\Omega_T)$ shows that this convergence also holds in $L^\rho(\Omega_T)$ with $2<\rho<\beta+2$. We infer from $v_i^{(\eps)}\rightharpoonup v_i$ weakly in $L^2(\Omega_T)$ that
\begin{align*}
  u_i^{(\eps)}v_i^{(\eps)}\rightharpoonup u_iv_i
  \quad\mbox{weakly in }L^1(\Omega_T). 
\end{align*}
Moreover, we have
\begin{align*}
  \na u_i^{(\eps)}\rightharpoonup \na u_i
  &\quad\mbox{weakly in }L^r(\Omega_T), \\
  \pa_t u_i^{(\eps)}\rightharpoonup\pa_t u_i
  &\quad\mbox{weakly in }L^s(0,T;W^{1,s'}(\Omega)').
\end{align*}
The limit $\eps\to 0$ in the weak formulation of \eqref{1.u} gives
\begin{align*}
  \pa_t u_i - \sigma_i\Delta u_i + \diver(u_iv_i) = 0
  \quad\mbox{in }\Omega, \quad (\sigma_i\na u_i + u_iv_i)\cdot\nu=0
  \quad\mbox{on }\pa\Omega,\ t>0. 
\end{align*}

We know from \eqref{2.Leps} with $g=-\na p_i(u^{(\eps)})$ that $(\sqrt{\eps}\na v_i^{(\eps)})$ and $(v_i^{(\eps)})$ are bounded in $L^2(\Omega_T)$. Furthermore, $(p_i(u^{(\eps)}))$ is bounded in $L^{(\beta+2)/\beta}(\Omega_T)$. Thus, up to subsequences,
\begin{align*}
  \eps\na v_i^{(\eps)}\to 0 &\quad\mbox{strongly in }L^2(\Omega_T), \\
  p_i(u^{(\eps)})\to p_i(u) &\quad\mbox{strongly in }
  L^{\rho}(\Omega_T) \quad\text{for all }\rho <(\beta+2)/\beta.
\end{align*}
The limit $\eps\to 0$ in the weak formulation
\begin{align*}
  \int_0^T\int_\Omega(\eps\na v_i^{(\eps)}:\na\phi_i
  + v_i^{(\eps)}\cdot\phi_i)dxdt = -\int_0^T\int_\Omega
  p_i(u^{(\eps)})\diver\phi_i dxdt
\end{align*}
for $\phi_i\in C_0^\infty(\Omega)$ yields $v_i = -\na p_i(u)$ in the sense of distributions. We conclude that $u_i$ solves
\begin{align*}
  \pa_t u_i - \sigma_i\Delta u_i - \diver(u_i\na p_i(u)) = 0
  \quad\mbox{in }\Omega,\ t>0.
\end{align*}
Finally, the bound for $(u_i^{(\eps)})$ in $W^{1,s}(0,T;W^{1,s'}(\Omega)')\hookrightarrow C^0([0,T];W^{1,s'}(\Omega)')$ implies that the initial datum is satisfied in the sense of $W^{1,s'}(\Omega)'$. 

\subsection{Case $1<\beta<2$, $d=2$}

We know from Lemma \ref{lem.eta3} that $(u_i^{(\eps)})$ is bounded in $L^q(0,T;W^{1,q}(\Omega_T))$ with $q=4\beta/(\beta+2)>4/3$. Since $(u_i^{(\eps)})$ is bounded in $L^{2\beta}(\Omega_T)$ and $\beta>1$, the $L^2(\Omega_T)$ bound for $(v_i^{(\eps)})$ shows that $(u_i^{(\eps)}v_i^{(\eps)})$ is bounded in $L^s(\Omega_T)$ for some $s>1$. These two estimates give a uniform bound for the time derivative of $u_i^{(\eps)}$. Therefore, we can apply the Aubin--Lions lemma to find a subsequence that is not relabeled such that, as $\eps\to 0$,
\begin{align*}
  u_i^{(\eps)}\to u_i\quad\mbox{strongly in }L^q(\Omega_T),
\end{align*}
and this convergence also holds in $L^\rho(\Omega_T)$ for $2<\rho<2\beta$. We infer that
\begin{align*}
  u_i^{(\eps)}v_i^{(\eps)}\rightharpoonup u_iv_i\quad
  \mbox{weakly in }L^1(\Omega_T).
\end{align*}
Moreover, Lemma \ref{lem.eta3} implies that, for a subsequence,
\begin{align*}
  \na u_i^{(\eps)}\rightharpoonup\na u_i\quad\mbox{weakly in }
  L^q(\Omega_T).
\end{align*}
These convergences are sufficient to pass to the limit in the weak formulation of \eqref{1.u}. The identification $v_i=\na p_i(u)$ is proved as in the previous subsection, taking into account that $p_i(u^{(\eps)})\to p_i(u)$ strongly in $L^\rho(\Omega_T)$ for $\rho<2$. 

\subsection{Case $\beta\ge 2$, $d\ge 1$}

Recall that $(u_i^{(\eps)})$ is bounded in $L^{\beta(d+2)/d}(\Omega_T)$. Since $\beta(d+2)/d>2$, we can proceed as in the previous subsection to find that (up to a subsequence) $u_i^{(\eps)}v_i^{(\eps)}\rightharpoonup u_iv_i$ weakly in $L^1(\Omega_T)$ as $\eps\to 0$. We have, for a subsequence, $u_i^{(\eps)}\to u_i$ strongly in $L^\rho(\Omega_T)$ for $2<\rho<\beta(d+2)/d$ and hence $p_i(u^{(\eps)})\to p_i(u)$ strongly in $L^{\rho/\beta}(\Omega_T)$. The only difference to the case $\beta<2$ is that we do not have a gradient bound for $u_i^{(\eps)}$ anymore. Therefore, we pass to the limit $\eps\to 0$ in
\begin{align*}
  -\int_0^T\int_\Omega u_i^{(\eps)}\Delta\phi_i dxdt
  \quad\mbox{instead of}\quad
  \int_0^T\int_\Omega\na u_i^{(\eps)}\cdot\na\phi_i dxdt
\end{align*}
for suitable test functions $\phi_i$ satisfying $\na\phi_i\cdot\nu=0$ on $\pa\Omega$. The remaining proof is similar to those of the previous subsections.


\section{Numerical simulation}\label{sec.num}

We present numerical simulations to explore how the solutions' behavior varies with the exponent $\beta$. We use the software Netgen/NGSolve \cite{Sch14} to solve equations \eqref{1.u}--\eqref{1.bic}. We choose $\Omega=(0,1)^2$, $n=3$, $\sigma_i=0.1$ ($i=1,2,3$), and $\eps=0.001$. The coefficient matrix
\begin{align*}
  (a_{ij}) = \begin{pmatrix}
  5 & 1 & 1 \\ 1 & 1 & 0.5 \\ 1 & 0.5 & 0.5 \end{pmatrix}
\end{align*}
is positive definite. The initial datum is chosen as
\begin{align*}
  u_i^0(x,y) = \exp\big(-100(x-x_i)^2 - 100(y-y_i)^2\big) + 0.5,
\end{align*}
where $(x_i,y_i) = (0.25\cdot i,0.25\cdot i)$ for $i=1,2,3$. The mesh size is $h=0.05$ and the time step size equals $\tau=4\cdot 10^{-5}$. Figure \ref{fig} shows the sum $u_1+u_2+u_3$ at different time steps. For all values of $\beta$, the solution tends to a constant steady state, whose value is determined by the mass of the initial data. We see that the decay is faster for larger values of $\beta$. The same phenomenon can be observed for smaller values of $\eps$ (not shown). Since $a_{11}$ is much larger than $a_{22}$ and $a_{33}$, $u_1$ diffuses faster than the other species.

\begin{figure}[htb]\centering
\includegraphics[width=150mm]{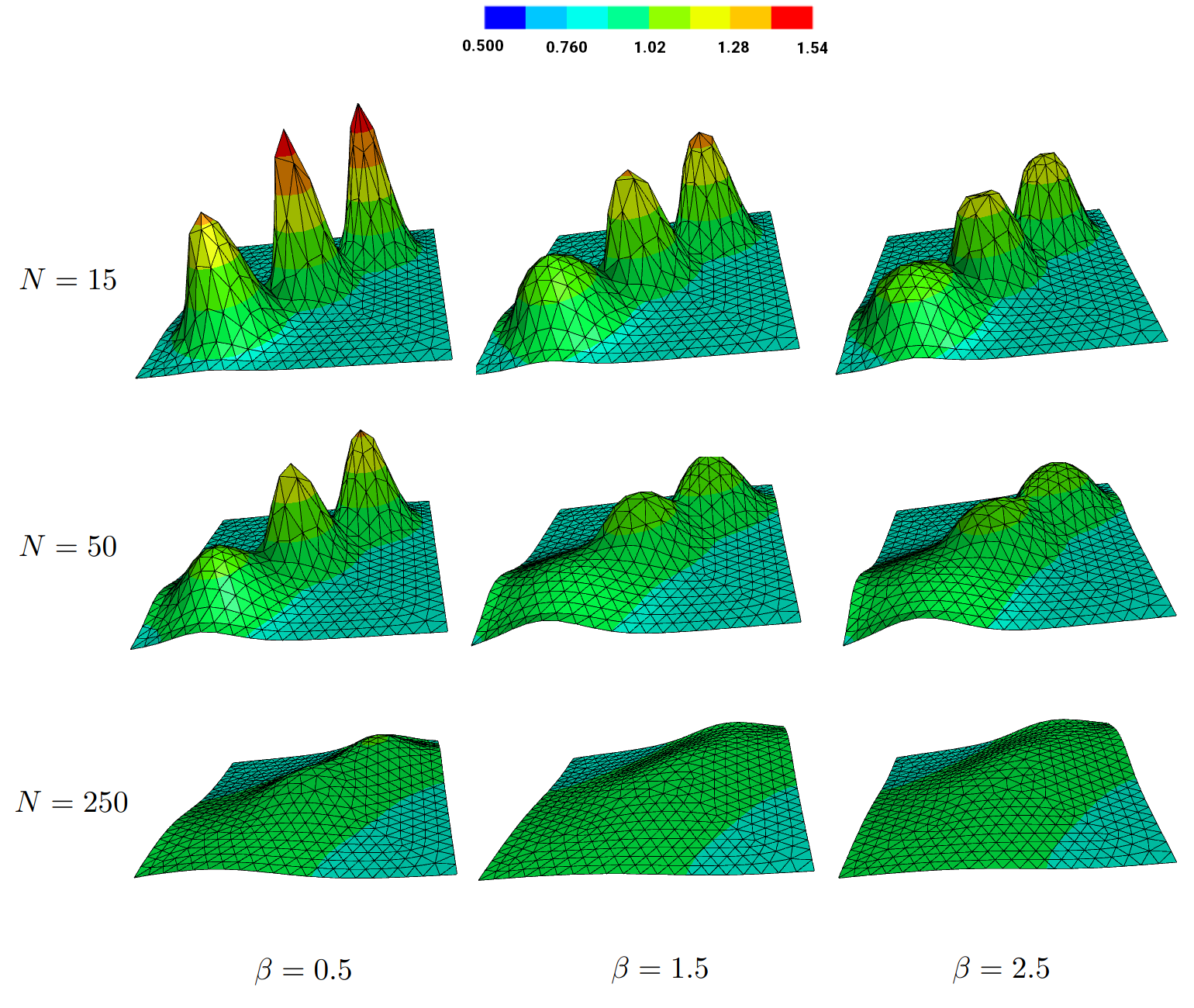}
\caption{Sum $u_1+u_2+u_3$ for $\beta=0.5$ (left column), $\beta=1,5$ (middle column), and $\beta=2.5$ (right column) at times $t=6\cdot 10^{-4}$ ($N=15$, top row), $t=2\cdot 10^{-3}$ ($N=50$, middle row), and $t=10^{-2}$ ($N=250$, bottom row).}
\label{fig}
\end{figure}


\begin{appendix}
\section{Modeling}\label{sec.model}

We derive equations \eqref{1.u}--\eqref{1.v} formally from a kinetic BGK (Bhatnagar--Gross--Krook) model, coupled to a Brinkman equation, following the approach of \cite{JPT25}. The kinetic equation for the distribution function $f_i^\delta(x,\xi,t)$, depending on the spatial variable $x\in\R^d$, velocity $\xi\in\R^d$, and time $t>0$, reads in the diffusion scaling as
\begin{align}\label{2.fdelta}
  \delta\pa_t f_i^\delta + \xi\cdot\na_x f_i^\delta
  + \na_x\Phi_i^\delta\cdot\na_\xi f_i^\delta
  = \frac{\nu_i}{\delta}(M_i(f_i^\delta)-f_i^\delta)
  \quad\mbox{in }\R^d,\ t>0,
\end{align}
for $i=1,\ldots,n$, where the scaling parameter $\delta>0$ relates to the Mach and Knudsen numbers and $\nu_i>0$ is the collision frequency of the $i$th species. Furthermore, $\Phi_i^\delta$ is a potential solving the nonlinear Brinkman law
\begin{align*}
  -\eps\Delta_x\Phi_i^\delta + \Phi_i^\delta = -P_i(\rho^\delta)
  = -\sum_{j=1}^n b_{ij}(\rho_j^\delta)^\beta
  \quad\mbox{in }\R^d,
\end{align*}
where $b_{ij}\in\R$, $\beta>0$, and the particle density $\rho_j^\delta$ is defined by
\begin{align*}
  \rho_j^\delta = \int_{\R^d}f_j^\delta d\xi, \quad j=1,\ldots,n.
\end{align*}
The Brinkman law is a generalization of the Darcy law $V_i=-\na_x P_i$, where $V_i$ is the partial velocity and $P_i$ is the partial pressure, by adding an elliptic regularization. The function $M_i(f_i^\delta)$ in \eqref{2.fdelta} is the Maxwell equilibrium distribution
\begin{align*}
  (M_i(f_i^\delta))(\xi) = \frac{\rho_i^\delta}{(2\pi)^{d/2}}
  \exp\bigg(-\frac{|\xi|^2}{2}\bigg),
\end{align*}
satisfying
\begin{align*}
  \int_{\R^d}M_i(f_i^\delta)d\xi = \rho_i^\delta, \quad
  \int_{\R^d}M_i(f_i^\delta)\xi d\xi = 0.
\end{align*}

Equations \eqref{1.u}--\eqref{1.v} are derived from a so-called Chapman--Enskog expansion:
\begin{align}\label{2.cee}
  f_i^\delta = f_i^0 + \delta g_i^\delta, \quad i=1,\ldots,n,
\end{align}
which defines the function $g_i^\delta$. The (formal) limit $\delta\to 0$ in \eqref{2.fdelta} gives $M_i(f_i^0) = f_i^0$, where
\begin{align*}
  (M_i(f_i^0))(\xi) = \frac{\rho_i}{(2\pi)^{d/2}}
  \exp\bigg(-\frac{|\xi|^2}{2}\bigg)
\end{align*}
and $\rho_i=\lim_{\delta\to 0}\rho_i^\delta$. We insert expansion \eqref{2.cee} into \eqref{2.fdelta},
\begin{align*}
  \delta\pa_t f_i^\delta + \xi\cdot\na_x(M_i(f_i^0)+\delta g_i^\delta)
  + \na_x\Phi_i^\delta\cdot\na_\xi(M_i(f_i^0)+\delta g_i^\delta)
  = -\nu_i g_i^\delta.
\end{align*}
The formal limit $\delta\to 0$ leads to
\begin{align}\label{2.g0}
  \xi\cdot\na_x M_i(f_i^0) + \na_x\Phi_i\cdot\na_\xi M_i(f_i^0)
  = -\nu_i g_i^0,
\end{align}
where we assume that $\Phi_i=\lim_{\delta\to 0}\Phi_i^\delta$ and $g_i^0 = \lim_{\delta\to 0}g_i^\delta$. 

We compute the moment equations. Integrating \eqref{2.fdelta} over $\xi\in\R^d$, the zeroth-order moment equation becomes
\begin{align*}
  \pa_t\int_{\R^d}f_i^\delta d\xi
  &+ \frac{1}{\delta}\int_{\R^d}M_i(f_i^0)\xi d\xi
  + \diver_x\int_{\R^d}g_i^\delta \xi d\xi
  + \frac{1}{\delta}\na_x\Phi_i^\delta\cdot\int_{\R^d}
  \na_\xi f_i^\delta d\xi \\
  &= -\frac{\nu_i}{\delta^2}\int_{\R^d}
  (M_i(f_i^\delta)-f_i^\delta)d\xi = 0.
\end{align*}
The second and fourth terms on the left-hand side vanish. The limit $\delta\to 0$ yields the mass conservation equation
\begin{align*}
  \pa_t\rho_i + \diver_x\int_{\R^d}g_i^0 \xi d\xi = 0.
\end{align*}
It remains to calculate $\int_{\R^d}g_i^0\xi d\xi$. For this, we multiply \eqref{2.g0} by $\xi$ and integrate over $\xi\in\R^d$:
\begin{align*}
  -\nu_i\int_{\R^d}g_i^0\xi d\xi &= \diver_x\int_{\R^d}
  M_i(f_i^0)(\xi\otimes\xi)d\xi + \na_x\Phi_i\cdot\int_{\R^d}
  \na_\xi M_i(f_i^0)\xi d\xi \\
  &= \na_x\int_{\R^d}M_i(f_i^0) d\xi
  - \na_x\Phi_i\cdot\int_{\R^d}M_i(f_i^0)d\xi
  = \na_x\rho_i - \rho_i\na_x\Phi_i.
\end{align*}
This shows that
\begin{align*}
  \pa_t\rho_i = -\frac{1}{\nu_i}\diver_x
  \big(\na_x\rho_i - \rho_i\na_x\Phi_i\big), \quad
  -\eps\Delta_x\Phi_i + \Phi_i = -\sum_{j=1}^n b_{ij}\rho_j.
\end{align*}
Setting $\sigma_i := 1/\nu_i$, $a_{ij} := b_{ij}/\nu_i$, $u_i:=\rho_i$, $v_i := \na_x\Phi_i/\nu_i$, and taking the gradient on both sides of the Brinkman law, we obtain \eqref{1.u}--\eqref{1.v}.

\section{Elliptic Regularity}\label{sec:appEllReg}
We give a brief justification of the elliptic regularity result in Assumption (A4). For second-order equations and $p\ge 2$, this result follows from \cite{Mey63}. The case $1<p<2$ is derived from a duality argument (see, e.g., the proof of Theorem 9.9, step (v) \cite{GiTr01}). Higher-order equations with right-hand side in $L^p(\Omega;\R^d)$ and $1<p<\infty$ are treated in \cite{ADN59}. In the whole-space setting, the elliptic regularity result can be obtained from the Fourier-transformed equation, related to \eqref{1.A4},
\begin{align*}
  (\eta|\xi|^{2m} + |\xi|^2 + 1)\widehat{v}(\xi) 
  = \mathrm{i}\xi\cdot\widehat{f}(\xi).
\end{align*}
This gives that
\begin{align*}
  |\widehat{v}(\xi)| 
  = \frac{|\xi\cdot\widehat{f}(\xi)|}{\eta|\xi|^{2m} + |\xi|^2 + 1}
  \le \frac{|\xi\cdot\widehat{f}(\xi)|}{|\xi|^2 + 1}, \quad
  |\widehat{\pa_{x_j}v(\xi)}|
  = |\mathrm{i} \xi_j\widehat{f}(\xi)|
  \le \frac{|\xi_j\xi\cdot\widehat{f}(\xi)|}{|\xi|^2 + 1}.
\end{align*}
For $p\in [1,2]$, we apply the Hausdorff--Young inequality to conclude that $\|\widehat v\|_{L^{p'}(\R^d)} \leq \|\widehat f\|_{L^{p'}(\R^d)}\leq C\|f\|_{L^p(\R^d)}$. Thus, the $W^{1,p}(\R^d)$ norm of $v$ is bounded uniformly in $\eta$ after another application of the Hausdorff--Young inequality. The case $p> 2$ again follows by duality. In bounded domains, we may first establish interior estimates in cubes, then derive an estimate near the boundary (by flattening), and finally combine both estimates and use a finite covering argument to achieve a global estimate. Observe that the higher-order term only improves the coercivity; it does not worsen the $W^{1,p}(\Omega)$ behavior.

\end{appendix}


\end{document}